\documentclass[english,15pt]{article}
\usepackage[T1]{fontenc}
\usepackage{geometry}
\usepackage{float}
\usepackage{mathrsfs}
\usepackage{amsmath}
\usepackage{amsthm}
\usepackage{amssymb}
\usepackage{graphicx}
\usepackage{hyperref}
\hypersetup{
    colorlinks=true,
    linkcolor=blue,
    filecolor=magenta,      
    urlcolor=cyan,
    citecolor = blue,
}

\newcounter{hypA}

\newcounter{hypH}

\usepackage{bbm} 
\usepackage{graphicx}
\usepackage{subcaption}
\usepackage{makeidx}
\usepackage{amssymb}
\usepackage[utf8]{inputenc}
\usepackage{booktabs}
\usepackage{enumitem}
\usepackage{algorithm}
\usepackage{algpseudocode}
\graphicspath{{../figures/}}

\providecommand{\U}[1]{\protect\rule{.1in}{.1in}}

\newtheorem{prop}{Proposition}[section]
\newtheorem{cor}[prop]{Corollary}

\newtheorem{lem}[prop]{Lemma}

\newtheorem{theo}[prop]{Theorem}

\newcommand{\RR}{\mathbb{R}}

\newcommand{\Na}{ {\cal N }}

\newcommand{\E}{\mathbb{E}}
\newcommand{\Var}{\textnormal{\textrm{Var}}}
\newcommand{\Cov}{\textnormal{\textrm{Cov}}}
\newcommand{\Lip}{\textnormal{\textrm{Lip}}}

\def \RR{\mathbb{R}}

\usepackage[textwidth=2.5cm, textsize=scriptsize]{todonotes}

\newcommand{\vertiii}[1]{{\left\vert\kern-0.25ex\left\vert\kern-0.25ex\left\vert #1
    \right\vert\kern-0.25ex\right\vert\kern-0.25ex\right\vert}}

\begin{document}


\begin{center}
{\Large \textbf{Sequential Entropic Optimal Transport for Bayesian Inverse Problems via Multilevel Monte Carlo}}

\vspace{0.5cm}
JASON BEH$^{1}$,  PIERRE DEL MORAL$^{2}$ \& 
AJAY JASRA$^{1}$ 

{\footnotesize $^{1}$Computing \& Mathematical Sciences Division,  Mohamed Bin Zayed University of Artificial Intelligence,  Abu Dhabi, UAE.}\\
{\footnotesize $^{2}$Centre de Recherche Inria Bordeaux Sud-Ouest, Talence, FR.}\\
{\footnotesize E-Mail:\,} 
\texttt{\emph{\footnotesize {Jason.KaiXun,Ajay.Jasra}@mbzuai.ac.ae, pierre.del\_moral@inria.fr}}

\begin{abstract}
We consider the problem of Bayesian estimation
associated to inverse problems governed by partial differential equations (PDEs). In most cases of practical
interest, one must use numerical solvers to implement computational methods based on Monte Carlo methods.
This inevitably leads to an adaptation of multilevel Monte Carlo (MLMC) methods to reduce computational efforts to achieve a pre-specified mean square error (MSE). 
MLMC is based on sampling coupled versions of Bayesian posteriors associated to increasingly precise numerical solves of the PDE.
In this article
we present a novel method to implement MLMC
based upon sequential entropic optimal transport couplings 
learned with light Schr\"odinger
bridges, and correct the transport with Metropolis--Hastings steps that preserve the pairing. We bound the mean square error associated to a version of our new estimation method
and show, for an elliptic inverse problem, that 
to achieve a MSE of $\mathcal{O}(\delta^2)$ ($\delta >0$)
one can have an optimal cost of $\mathcal{O}(\delta^{-2})$.
The results feature explicit bounds on the effect of the fitted transport and of the Metropolis--Hastings correction. 
From a numerical perspective, 
on elliptic, Darcy-flow and Navier--Stokes problems, the method reduces the variance of the
most precise solver by a factor of five to ten and is faster than multilevel sequential Monte Carlo 
\cite{mlsmc}
at tight accuracies.
\\
\bigskip
\noindent\textbf{Keywords:} Entropic Optimal Transport, 
Light Schr\"odinger Bridges, 
Bayesian Inverse Problems,  Multilevel Monte Carlo.
\end{abstract}

\end{center}




\section{Introduction}\label{sec:background}

Inverse problems and their Bayesian counterparts
are found in a large class of practical applications,
which include physics, engineering, finance, and applied mathematics; see e.g.~\cite{stuart} and the references therein. In many cases, these types of models are related to partial differential equations (PDEs), which require
numerical solution, and, in addition, the associated Bayesian estimation problem relates to a posterior density that can only be computed point-wise and with numerical
error associated to the PDE solver. This naturally means that one must resort to computational methods based on, for example, Markov chain simulation, for which there have been numerous contributions; see for example \cite{smc_elliptic,dkst,kbj,mlsmc}.

The intention of this article is to introduce a new methodology,
using machine learning methods, which can improve in terms of accuracy
and speed, when related to these afore-mentioned Markov chain methods.
Our approach will be based on entropic optimal transport and to that end, we give some mathematical background to provide context going forward.

\subsection{Discretized Bayesian inverse problems}\label{sec:problem}

Let $\pi$ be a probability measure on a measurable space $(\mathsf{X},\mathscr{X})$ with $\mathsf{X}\subseteq\mathbb{R}^d$.
For a $\pi$-integrable $\varphi:\mathsf{X}\rightarrow\mathbb{R}$ we wish to compute
$$
\pi(\varphi) := \int_{\mathsf{X}}\varphi(x)\pi(dx).
$$
We assume that $\pi$ has a density with respect to a $\sigma$-finite measure $dx$ on $(\mathsf{X},\mathscr{X})$,
$$
\pi(dx) = \frac{\gamma(x)dx}{Z},
$$
where $\gamma:\mathsf{X}\rightarrow\mathbb{R}^+$ and $Z=\int_{\mathsf{X}}\gamma(x)dx$.

We are interested in the case where $\gamma$ cannot be evaluated, but approximations $\gamma_l:\mathsf{X}\rightarrow\mathbb{R}^+$,
$l\in\mathbb{N}_0$, can. We assume that, for each $l\in\mathbb{N}_0$, $Z_l=\int_{\mathsf{X}}\gamma_l(x)dx$ is finite and positive, so that
$$
\pi_l(dx) = \frac{\gamma_l(x)dx}{Z_l}
$$
is a probability measure, and that for every $\pi$-integrable $\varphi:\mathsf{X}\rightarrow\mathbb{R}$,
$$
\lim_{l\rightarrow\infty}\pi_l(\varphi) = \pi(\varphi),
$$
where $\pi_l(\varphi) := \int_{\mathsf{X}}\varphi(x)\pi_l(dx)$. This is the situation in Bayesian inverse problems, the class of problem of interest, where
$\gamma_l$ involves the numerical solution of a PDE at resolution $l$.

\subsection{Background on entropic optimal transport}\label{sec:eot}


Our main objective is to sample couplings of the targets
$\pi_l$ and $\pi_{l-1}$ as we will now describe.
Let $\mu$ and $\nu$ be probability measures on $(\mathbb{R}^d,\mathscr{B}(\mathbb{R}^d))$ with Lebesgue densities, also
denoted $\mu$ and $\nu$, and let $\Pi(\mu,\nu)$ be the set of their couplings; that is the collection of probability measures that have
marginals $\mu$ and $\nu$ respectively.
For a probability density $\xi$ on
$\mathbb{R}^p$, $H(\xi)= -\int_{\mathbb{R}^p}\xi(x)\log\{\xi(x)\}dx$ is its differential entropy. For $\varepsilon>0$, the
entropic optimal transport (EOT) problem is
\begin{align} \label{eq:eot-prob}
\inf_{\xi\in\Pi(\mu,\nu)}\left\{ \frac{1}{2}\int_{\mathbb{R}^{2d}}\|x-y\|_2^2\xi(x,y)\,dx\,dy - \varepsilon H(\xi)\right\},
\end{align}
where $\|\cdot\|_2$ is the Euclidean norm.
As noted by several authors (e.g.,~\cite{koro2}), the optimal coupling can be written as
$$
\xi^{\star}(x,y) = \mu(x) \frac{\exp\left\{\tfrac{\langle x,y\rangle }{\varepsilon}\right\}v^{\star}(y)}
{\int_{\mathbb{R}^d}\exp\left\{\tfrac{\langle x,y\rangle }{\varepsilon}\right\}v^{\star}(y)dy},
$$
where the potential $v^{\star}:\mathbb{R}^d\rightarrow\mathbb{R}^+$ is unknown.
For a potential $v:\mathbb{R}^d\rightarrow\mathbb{R}^+$ write $Z_{v}(x):=\int_{\mathbb{R}^d}
\exp\left\{\tfrac{\langle x,y\rangle }{\varepsilon}\right\}v(y)dy$. To approximate $v^{\star}$, let $v_{\theta}$ be a
potential with finite-dimensional parameter $\theta\in\Theta\subseteq\mathbb{R}^{d_{\theta}}$. In \cite{koro1}, $\theta$ is
fitted by minimizing
$$
\mathcal{L}(\theta) = \int_{\mathbb{R}^d}\log\{Z_{v_\theta}(x)\}\mu(x)dx - 
\int_{\mathbb{R}^d}\log\{v_\theta(y)\}\nu(y)dy ,
$$
with the Gaussian-mixture potential
$$
v_{\theta}(y)  = \sum_{j=1}^J \vartheta_j\phi_d(y;m_j,\Sigma_j),
$$
where $\phi_d(y;m,\Sigma)$ is the $d$-dimensional Gaussian density with mean $m$ and covariance matrix $\Sigma$,
$\vartheta_j\geq 0$, $\sum_{j=1}^J\vartheta_j = 1$ and $\theta=(m_{1:J},\Sigma_{1:J},\vartheta_{1:J-1})$. With this choice
$Z_{v_\theta}(x)$ is available in closed form and the conditional law of $y$ given $x$ is a Gaussian mixture that can be
sampled exactly, which gives a fast, though possibly biased, approximation of $\xi^{\star}$.

\subsection{Related work, contributions and paper structure}

The main idea of this article is to use EOT to sample couplings
of $\pi_l$ and $\pi_{l-1}$ within the context of 
Multilevel Monte Carlo \cite{cgst,giles,giles_acta,heinrich} (MLMC).
MLMC writes $\pi_L(\varphi)$ as a telescoping sum of corrections between
consecutive levels and estimates each one using coupled samples. The utility of doing so is that for sufficiently good couplings one gains
a variance reduction and choosing the number of simulated samples appropriately, one can reduce the computational cost to achieve a mean square error (MSE), versus just considering sampling from $\pi_L$ say.
The `good couplings' we mention above will be obtained by using
EOT.

In the context of MLMC, and for Bayesian inverse problems, couplings have been built with Markov chains \cite{dkst,hss} or with sequential Monte
Carlo samplers \cite{dmdj}, as in the multilevel sequential Monte Carlo (MLSMC) method of \cite{mlsmc}, later
extended with likelihood-informed proposals \cite{mlsmc_dili} and with tempering \cite{ls2mc}. Sequential Monte Carlo
methods of this type have been applied, for instance, to inference for the Navier--Stokes equations \cite{kbj},
Lagrangian data assimilation with unknown observation locations \cite{rbcjk} and, more recently, zeroth-order
stochastic bilevel optimization \cite{csss}.
Instead of reweighting samples of one distribution to represent another, 
optimal transport moves samples, through a deterministic map or a Markov kernel that pushes the first distribution onto the second; see \cite{mmps} for an introduction. Transport maps have
been used to sample posteriors directly \cite{em}, also as compositions of maps between consecutive tempered
densities \cite{cd,zzcs}, to accelerate Markov chain Monte Carlo \cite{pm_tmcmc,pm_multi}, and to combine annealing with
multi-fidelity models \cite{sbm}. In multilevel estimators they have been used to couple the levels \cite{hjs,mmg}, for
instance by pushing a common reference sample through a map fitted to each level. In sequential Monte Carlo, normalizing flows
\cite{aft,craft}, controlled transitions \cite{csmc} and Schr\"odinger bridges \cite{sbs} have been learned
between consecutive tempered distributions, with importance weights to correct the transport.

Rather than a deterministic map, entropic optimal transport \cite{leonard,nutz} gives a coupling of two distributions: the joint law with the prescribed marginals that minimizes the expected squared distance between the coupled points plus an entropy penalty. Its conditional laws act as a stochastic transport map, and the coupled pairs are close when the two distributions are close. The coupling is computed by Sinkhorn iterations
between empirical measures \cite{cuturi,pc}, whose stability is studied in \cite{dmj}.
The light Schr\"odinger bridge \cite{koro2,koro1} parametrizes the potential by a Gaussian mixture, which makes the
coupling cheap to fit from samples and exact to simulate.

In this article we present two distinct methods for implementing
MLMC via EOT. In the first, we use a light Schr\"odinger bridge
to learn the EOT between each $\pi_l$ and $\pi_{l-1}$ and combine this with a type of particle simulation
method. In the second, this approach is augmented with Metropolis-Hastings steps.  
We then provide a mathematical analysis of our approach in terms of an MSE analysis. We bound the MSE associated to a version of our new estimation method
and show, for an elliptic inverse problem, that 
to achieve a MSE of $\mathcal{O}(\delta^2)$ ($\delta >0$)
one can have an optimal cost of $\mathcal{O}(\delta^{-2})$, depending on the cost of the solver.
The results feature explicit bounds on the effect of the fitted transport and of the Metropolis--Hastings correction. 
We then consider three numerical examples associate to
elliptic, Darcy-flow and Navier--Stokes problems. We show that the second method reduces the variance of the
most precise solver by a factor of five to ten and is faster than multilevel sequential Monte Carlo 
\cite{mlsmc} at tight accuracies.


This article is structured as follows. In section \ref{sec:approach} our methodology is detailed.
In Section \ref{sec:analysis} we present our mathematical results, with some proofs deferred to the appendix. Section \ref{sec:numerics}
details our numerical results.

\section{Multilevel Estimation with Learned Entropic Couplings}\label{sec:approach} 

\subsection{Notation}\label{sec:notation}
Throughout, $\mathsf X\subseteq\RR^d$ is the state space and $x$ a point of $\mathsf X$; $\|\cdot\|_2$ and
$\langle\cdot,\cdot\rangle$ are the Euclidean norm and inner product. Probability measures with a Lebesgue density are
identified with that density. For a probability measure $\mu$ and a function $\varphi$, $\mu(\varphi)=\int\varphi\,d\mu$ and
$\Var_\mu(\varphi)$ is the variance of $\varphi$ under $\mu$; $\E_\xi$ is the expectation under a law $\xi$. A function $\varphi$ is Lipschitz if
$\|\varphi\|_{\Lip}=\sup_{x\neq x'}|\varphi(x)-\varphi(x')|/\|x-x'\|_2<\infty$. For a Markov kernel $K$,
$\mu K(dy)=\int\mu(dx)K(x,dy)$ and $(\mu\otimes K)(dx,dy)=\mu(dx)K(x,dy)$.

For probability measures $\mu$ and $\nu$, $\Pi(\mu,\nu)$ is the set of couplings, i.e., of laws on $\mathsf X\times\mathsf X$
with marginals $\mu$ and $\nu$, and $W_2^2(\mu,\nu)=\inf_{\xi\in\Pi(\mu,\nu)}\E_\xi\|X-Y\|_2^2$. We write
$\|\mu-\nu\|_{\mathrm{TV}}=\sup_A|\mu(A)-\nu(A)|$, $\mathrm{KL}(\mu\|\nu)=\int\log(d\mu/d\nu)\,d\mu$, and
$H(\xi)=-\int\xi\log\xi$ for the differential entropy of a density $\xi$. $\Na(m,\Sigma)$ is the Gaussian law with mean $m$
and covariance $\Sigma$ and $\phi_d(\cdot\,;m,\Sigma)$ its density on $\RR^d$. For nonnegative sequences,
$a_l\lesssim b_l$ means $a_l\le Cb_l$ for a constant $C$ that does not depend on $l$.

\subsection{Multilevel transport with light Schr\"odinger bridges}\label{sec:transport}

We approximate $\pi_L(\varphi)$ through the telescoping sum
$$
\pi_L(\varphi) = \pi_0(\varphi) + \sum_{l=1}^L\left\{\pi_l(\varphi) - \pi_{l-1}(\varphi)\right\}.
$$
If one can sample from couplings $\xi_l\in\Pi(\pi_{l-1},\pi_{l})$ whose transport cost
$$
\int_{\mathbb{R}^{2d}}\|x-y\|_2^2\xi_l(x,y)\,dx\,dy
$$
is small, each correction can be estimated from few samples, and a prescribed mean square error is reached at lower
cost than by sampling from $\pi_L$ alone; see \cite{mlsmc} and the references therein. The method of \cite{mlsmc}
achieves this with a combination of sequential and Markov chain Monte Carlo. We instead learn, at each level, an
entropic coupling of consecutive targets with the light Schr\"odinger bridge of \cite{koro1}.

Throughout the paper, the numbers of particles satisfy $1\le N_L\le\cdots\le N_0<\infty$, so that every particle
at level $l$ has a parent at level $l-1$. Given the Gaussian-mixture potentials $v_\theta$ of Section~\ref{sec:eot}, the
parameter of the transition from level $l$ to level $l+1$ is fitted by stochastic gradient descent on the
self-normalized empirical loss
\begin{equation}\label{eq:loss}
\mathcal{L}^{N_l}(\theta) = \frac{1}{N_l}\sum_{i=1}^{N_l} \log\{Z_{v_\theta}(X_l^i)\}
- \frac{\frac{1}{N_l}\sum_{i=1}^{N_l}\frac{\gamma_{l+1}(X_l^i)}{\gamma_l(X_l^i)}
\log\{v_\theta(X_l^i)\}
}{\frac{1}{N_l}\sum_{i=1}^{N_l}\frac{\gamma_{l+1}(X_l^i)}{\gamma_l(X_l^i)}},
\end{equation}
and the particles are moved with the kernel
\begin{equation}\label{eq:ktheta}
  K_\theta(x,dy)=\frac{\exp\{\langle x,y\rangle/\varepsilon\}\,v_\theta(y)}{Z_{v_\theta}(x)}\,dy ,
\end{equation}
with $\varepsilon=\varepsilon_{l+1}$ the entropic parameter of the transition from level $l$ to level $l+1$. The resulting
method is Algorithm~\ref{alg:lsb}.
\begin{algorithm}[htbp]
\caption{Multilevel estimation with light Schr\"odinger bridge transport}\label{alg:lsb}
\begin{algorithmic}[1]
\Require horizon $L$, sample sizes $N_0\ge\dots\ge N_L$, entropic parameters $\varepsilon_1,\dots,\varepsilon_L$
\State sample $X_0^1,\dots,X_0^{N_0}$ i.i.d.\ from $\pi_0$
\For{$l=0,\dots,L-1$}
  \State compute $\theta_{l+1}$ by stochastic gradient descent on $\mathcal L^{N_l}(\theta)$ in~\eqref{eq:loss}
  \For{$i=1,\dots,N_{l+1}$}
    \State sample $X_{l+1}^i\sim K_{\theta_{l+1}}(X_l^i,\cdot)$, with $K_\theta$ in~\eqref{eq:ktheta} and $\varepsilon=\varepsilon_{l+1}$
  \EndFor
\EndFor
\State \Return $\widehat\pi_L(\varphi)$ in~\eqref{eq:estimator}
\end{algorithmic}
\end{algorithm}
The particles at level $l$ have law $\pi^\theta_l$, where $\pi^\theta_0=\pi_0$ and
$\pi^\theta_{l+1}=\pi^\theta_lK_{\theta_{l+1}}$.

The estimator of $\pi_L(\varphi)$ is
\begin{equation}\label{eq:estimator}
\widehat{\pi}_L(\varphi) = \frac{1}{N_0}\sum_{i=1}^{N_0}\varphi(X_0^i) + \sum_{l=1}^L
\frac{1}{N_l}\sum_{i=1}^{N_l}\{\varphi(X_l^i) -\varphi(X_{l-1}^i)\},
\end{equation}
in which each correction is a paired difference along the transport. Line 1 of Algorithm~\ref{alg:lsb} assumes
i.i.d.\ samples from $\pi_0$, as
in~\cite{mlsmc}. This is realistic when the level-0 solver is coarse enough for rejection sampling from the prior to be
efficient; otherwise $\pi_0$ can be approximated by a tempered sequential Monte Carlo sampler, as in
Section~\ref{sec:numerics}. The same mechanism gives approximate samples from any sequence of related probability
measures on a common state space.

\subsection{Transport followed by Metropolis--Hastings}\label{sec:tmh}

The particles produced in line 5 of Algorithm~\ref{alg:lsb} have law $\pi_{l+1}^\theta =  \pi^\theta_lK_{\theta_{l+1}}$ rather
than $\pi_{l+1}$, and this error accumulates over levels. A simple correction is to follow each transport step by a few
iterations of a Markov kernel $M_{l+1}$ that leaves $\pi_{l+1}$ invariant, started at the transported particle. We call
the resulting method TransportMH; it is given in Algorithm~\ref{alg:tmh}.
\begin{algorithm}[htbp]
\caption{TransportMH: transport followed by Metropolis--Hastings}\label{alg:tmh}
\begin{algorithmic}[1]
\Require horizon $L$, sample sizes $N_0\ge\dots\ge N_L$, entropic parameters $\varepsilon_1,\dots,\varepsilon_L$, numbers of
sweeps $n_1,\dots,n_L\ge0$, Markov kernels $M_1,\dots,M_L$ with $\pi_lM_l=\pi_l$
\State sample $X_0^1,\dots,X_0^{N_0}$ i.i.d.\ from $\pi_0$
\For{$l=0,\dots,L-1$}
  \State compute $\theta_{l+1}$ by stochastic gradient descent on $\mathcal L^{N_l}(\theta)$ in~\eqref{eq:loss}
  \For{$i=1,\dots,N_{l+1}$}
    \State sample $\widetilde X_{l+1}^i\sim K_{\theta_{l+1}}(X_l^i,\cdot)$
    \State apply $n_{l+1}$ steps of $M_{l+1}$ started at $\widetilde X_{l+1}^i$, and call the result $X_{l+1}^i$
  \EndFor
\EndFor
\State \Return $\widehat\pi_L(\varphi)$ in~\eqref{eq:estimator}
\end{algorithmic}
\end{algorithm}
The estimator is $\widehat{\pi}_L(\varphi)$ in~\eqref{eq:estimator}, computed with the particles after the
Metropolis--Hastings (MH) step.

In our implementation $M_{l+1}$ is a preconditioned Crank--Nicolson Metropolis--Hastings (pCN-MH) kernel. Let
$(m,AA^\top)$ be the mean and covariance of $\{\widetilde X_{l+1}^i\}_{i\le N_{l+1}}$. The proposal is
\[
  x'=m+A\big\{\rho\,A^{-1}(x-m)+(1-\rho^2)^{1/2}U\big\},\qquad U\sim\Na(0,I_d),
\]
which is reversible with respect to $\Na(m,AA^\top)$, and it is accepted with probability
\[
  \min\Big\{1,\ \frac{\gamma_{l+1}(x')\,\phi_d(x;m,AA^\top)}{\gamma_{l+1}(x)\,\phi_d(x';m,AA^\top)}\Big\},
\]
so $M_{l+1}$ leaves $\pi_{l+1}$ invariant for any $(m,A)$ and $\rho\in[0,1)$. The step and the number of sweeps adapt to the difficulty of the transition, measured by the effective sample size
(ESS) fraction $e_{l+1}$ of the weights $\gamma_{l+1}/\gamma_l$ on the level-$l$ cloud. With
$t_{l+1}=\min\{1,\max(0,\log(e_{l+1}/0.05)/\log10)\}$, which is 0 if $e_{l+1}\le5\%$ and 1 if $e_{l+1}\ge50\%$,
\[
  n_{l+1}=\mathrm{round}(30-27t_{l+1}),\qquad
  \log(1-\rho_{l+1}^2)=(1-t_{l+1})\log(1-0.8^2)+t_{l+1}\log(1-S_{l+1}^2),
\]
where $S_{l+1}=\{(\varepsilon_{l+1}^2+4)^{1/2}-\varepsilon_{l+1}\}/2$ is the correlation of the EOT coupling of $\Na(0,I_d)$ with
itself at parameter $\varepsilon_{l+1}$~\cite{adm,jmpc,mgm}. After a hard transition the chain makes 30 sweeps of large
moves ($\rho=0.8$); once the weights are close to uniform it makes 3 sweeps of moves at the scale of the transport noise,
since $1-S_{l+1}^2=\varepsilon_{l+1}S_{l+1}$.

Four remarks are in order.
\begin{itemize}
\item With $n_{l+1}=0$ for all $l$, Algorithm~\ref{alg:tmh} reduces to Algorithm~\ref{alg:lsb}.
\item Since $M_{l+1}$ is $\pi_{l+1}$-invariant, $\|\mu M_{l+1}-\pi_{l+1}\|_{\mathrm{TV}}\le\|\mu-\pi_{l+1}\|_{\mathrm{TV}}$ for
  any $\mu$, so the MH step can only reduce the distance of the particle law to $\pi_{l+1}$. Finitely many sweeps do not
  remove this distance, so $\widehat{\pi}_L(\varphi)$ remains biased for $\pi_L(\varphi)$.
\item Each chain starts at the transported particle, so $X_{l+1}^i$ stays paired with its parent $X_l^i$. When
  $\rho$ is close to one the pairing is barely affected; larger moves trade a weaker pairing, and hence a larger
  variance of $\varphi(X_{l+1}^i)-\varphi(X_l^i)$, for a smaller bias.
\item The transition from level $l$ to level $l+1$ costs $N_l$ evaluations of $\gamma_{l+1}$ for the training weights,
  $N_{l+1}(1+n_{l+1})$ for the transported particles and the MH sweeps, and the stochastic gradient fit, which needs no evaluation of
  $\gamma_{l+1}$.
\end{itemize}



\section{Mathematical Analysis}\label{sec:analysis}
We analyze the paired estimator~\eqref{eq:estimator}. Section~\ref{sec:mse} bounds its mean square error when each
transition is an exact coupling of consecutive targets, in terms of the transport costs between levels, and derives
the multilevel rates and complexity. Section~\ref{sec:elliptic-theory} applies these bounds to posteriors with bounded
misfit, such as the elliptic problem of Section~\ref{sec:num-1d}, and extends them to the fitted kernels, with and
without the Metropolis--Hastings correction. The proofs of the results stated without proof are collected in
Appendix~\ref{app:proofs}.

The results concern Algorithms~\ref{alg:lsb} and~\ref{alg:tmh} and idealized versions of them, as follows.
\begin{itemize}
\item Theorem~\ref{thm:mse} and Corollaries~\ref{cor:rates} and~\ref{cor:elliptic} analyze Algorithm~\ref{alg:lsb} with
  the fitted kernel $K_{\theta_{l+1}}$ of line~5 replaced by an exact coupling kernel $K_{l+1}$, which moves the particles
  of level $l$ exactly onto $\pi_{l+1}$, and with line~3 omitted. This oracle version isolates the effect of the
  transport cost between levels.
\item Corollary~\ref{cor:fitted} analyzes Algorithm~\ref{alg:lsb} itself, with the fit of line~3 computed on samples
  independent of the particles.
\item Corollary~\ref{cor:fittedmh} analyzes Algorithm~\ref{alg:tmh}, with the same fits and Metropolis--Hastings kernels in
  line~6 that contract in total variation.
\end{itemize}
We first state the assumption under which the oracle version is analyzed in Section~\ref{sec:mse}.
\begin{enumerate}[label=(A\arabic*),ref=(A\arabic*),leftmargin=3.2em,series=assumptions]
\item\label{ass:A1} For each $l\geq 1$ there exists a Markov kernel $K_l$ such that $\xi_l(dx,dy)=\pi_{l-1}(d x)K_l(x, dy)$ has second
  marginal $\pi_l$, that is, $\pi_{l-1}K_l=\pi_l$ and $\xi_l\in\Pi(\pi_{l-1},\pi_l)$.
\end{enumerate}

\subsection{MSE under exact couplings}\label{sec:mse}
Each particle index $i$ defines a lineage $X_0^i\to X_1^i\to\cdots\to X_{\ell_i}^i$,
where $\ell_i=\max\{l:N_l\ge i\}$. Exactly $N_l$ lineages reach level $l$, and a lineage that reaches level $l$
reaches all levels $k<l$. For Markov kernels $K_1,K_2,\dots$, exact or not, let $\pi^K_0=\pi_0$ and
$\pi^K_l=\pi^K_{l-1}K_l$. Each lineage is a Markov chain with initial law $\pi_0$ and transition kernels $K_1,K_2,\dots$,
so consecutive states $(X_{l-1}^i,X_l^i)$ have joint law $\pi^{K}_{l-1}\otimes K_l$, and distinct lineages are
independent. For $K_l=K_{\theta_l}$, $\pi^K_l$ is the law $\pi^\theta_l$ of Section~\ref{sec:approach}. Under \ref{ass:A1},
$\pi^K_l=\pi_l$ and $(X_{l-1}^i,X_l^i)\sim\xi_l$. Along a lineage define the increments
\[
  D_0=\varphi(X_0) \text{ and } D_l=\varphi(X_l)-\varphi(X_{l-1}) \text{ for } l\geq 1,
\]
and write $V_l=\Var(D_l)$ and $C_{kl}=\Cov(D_k,D_l)$.

If $\varphi\in L^1(\pi^K_l)$ for all $l$, then $\E\widehat \pi_L(\varphi)=\pi^K_L(\varphi)$. Under \ref{ass:A1},
$\E\widehat \pi_L(\varphi)=\pi_L(\varphi)$, so the bias with respect to $\pi(\varphi)$ is only the discretization error
$\pi_L(\varphi)-\pi(\varphi)$. Since the lineages are independent and identically distributed, for
$\varphi\in L^2(\pi^K_l)$ for all $l$,
\[
  \Var\big(\widehat \pi_L(\varphi)\big)=\sum_{l=0}^L\frac{V_l}{N_l}+2\sum_{0\le k<l\le L}\frac{C_{kl}}{N_k}.
\]
For $l\ge1$, let $\mathcal W_l:=\E_{\xi_l}\|X-Y\|_2^2$ be the transport cost of $\xi_l$ and $C_\varphi := \|\varphi\|_{\Lip}^2$.
If $\varphi$ is Lipschitz, then $V_l\le\E_{\xi_l}|\varphi(Y)-\varphi(X)|^2\leq C_\varphi\mathcal W_l$. Write
\[
  s_0:=\Var_{\pi_0}(\varphi)^{1/2},\qquad s_l:=(C_\varphi\,\mathcal W_l)^{1/2}\quad(l\ge1).
\]
The following theorem concerns the oracle version of Algorithm~\ref{alg:lsb}, in which line~5 uses the kernel $K_{l+1}$
of~\ref{ass:A1}. It shows that the mean square error of the estimator with respect to $\pi_L(\varphi)$ is
controlled by the transport costs $\mathcal W_l$ of the couplings: the covariances between levels, which are not zero
for the paired estimator, at most square the sum of the level standard deviations.
\begin{theo}\label{thm:mse}
Assume \ref{ass:A1}. Then for any Lipschitz $\varphi$ we have
\[
  \E\big[\{\widehat \pi_L(\varphi)-\pi_L(\varphi)\}^2\big]
  \le\sum_{l=0}^L\frac{s_l^2}{N_l}+2\sum_{0\le k<l\le L}\frac{s_k\,s_l}{N_k}
  \le\bigg(\sum_{l=0}^L\frac{s_l}{\sqrt{N_l}}\bigg)^2 .
\]
\end{theo}
\begin{proof}
Since $\E\widehat \pi_L(\varphi)=\pi_L(\varphi)$,
$$
\E[\{\widehat \pi_L(\varphi)-\pi_L(\varphi)\}^2]=\Var(\widehat \pi_L(\varphi))
= \sum_{l=0}^L\frac{V_l}{N_l}+2\sum_{0\le k<l\le L}\frac{C_{kl}}{N_k}.
$$
Here $V_0=s_0^2$, $V_l\le C_\varphi\mathcal W_l = s_l^2$ by the Lipschitz property of $\varphi$, and
$|C_{kl}|\le\sqrt{V_kV_l}\le s_ks_l$ by the Cauchy--Schwarz inequality, which gives the first inequality. The second
follows from $1/N_k\le1/\sqrt{N_kN_l}$ for $k<l$, since $N_l\le N_k$.
\end{proof}

We denote by $\alpha$, $\beta$ and $\zeta$ the rates of the discretization bias, the level variances and the cost
per sample: $|\pi_l(\varphi)-\pi(\varphi)|\lesssim h_l^\alpha$, $V_l\lesssim h_l^\beta$, and one
evaluation of $\gamma_l$ costs $\lesssim h_l^{-\zeta}$. The next result turns Theorem~\ref{thm:mse} into a complexity
statement: once the level variances, the bias and the cost decay at these rates, the classical allocation of the sample
sizes~\cite{cgst} applies to the paired estimator, even though its level corrections are not independent. Below let $h_l=h_02^{-l}$ a scale attached to level $l$, for instance the mesh size of
the solver that defines $\pi_l$.
\begin{cor}\label{cor:rates}
Assume \ref{ass:A1}, let $\varphi$ be Lipschitz, $h_l=h_02^{-l}$. Suppose that there exist constants $c_1$, $c_2$ and $c_3$ such that
$|\pi_l(\varphi)-\pi(\varphi)|\le c_1h_l^\alpha$, $V_l\le c_2h_l^\beta$ for $l\ge1$, and the cost is at most
$c_3\sum_lN_lh_l^{-\zeta}$. Then, with $c_\beta=2^{-\beta/2}/(1-2^{-\beta/2})$, we have 
\[
  \E\big[\{\widehat \pi_L(\varphi)-\pi(\varphi)\}^2\big]\le c_1^2h_L^{2\alpha}
  +\frac{V_0+2c_\beta\sqrt{V_0c_2}\,h_0^{\beta/2}}{N_0}+(1+2c_\beta)c_2\sum_{l=1}^L\frac{h_l^\beta}{N_l}.
\]
Moreover, if $\alpha\ge\frac12\min(\beta,\zeta)$, the choice $N_l\propto\delta^{-2}h_l^{(\beta+\zeta)/2}$ gives a mean
square error at most $\delta^2$ at cost
\begin{itemize}
    \item $O(\delta^{-2})$ for $\beta>\zeta$ (coarse levels),
    \item $O(\delta^{-2}\log^2\delta)$ for $\beta=\zeta$ (moderate levels), or
    \item $O(\delta^{-2-(\zeta-\beta)/\alpha})$ for $\beta<\zeta$ (fine levels).
\end{itemize}
\end{cor}
In Corollary~\ref{cor:rates}, $c_1$ and $c_3$ are properties of the problem and of the solver, $c_2$ depends also on the
couplings, and the user chooses the couplings and the $N_l$. Since $V_l\le\|\varphi\|_{\Lip}^2\mathcal W_l$, the level
variances decay at rate $\beta$ as soon as the transport costs do; the next section gives conditions for this.

\subsection{Posteriors with bounded misfit}\label{sec:elliptic-theory}

Theorem~\ref{thm:mse} does not require any particular form of the levels. We now consider posteriors on a compact
parameter space whose misfit is bounded and converges uniformly, the setting of~\cite[Section 4]{mlsmc}. Let
$\mathsf{X}=[-1,1]^d$ carry the uniform prior $\pi_{\mathrm{pr}}$, let $\mathcal G_l:\mathsf{X}\to\RR^{d_y}$ be the forward map
computed at level $l$ and $\mathcal G$ the exact one, and let $\gamma_l(x)=\exp\{-\Phi_l(x)\}$ with
$\Phi_l(x)=\|y-\mathcal G_l(x)\|_2^2/(2\sigma_\eta^2)$. We make the following assumption on the forward maps.
\begin{enumerate}[resume*=assumptions]
\item\label{ass:A2} There exist $R_{\mathcal G}<\infty$ and $c_{\mathcal G}<\infty$, independent of $x$ and $l$, such that
\begin{equation}\label{eq:fe}
  \sup_{x\in \mathsf{X}}\|\mathcal G_l(x)\|_2\le R_{\mathcal G},\qquad
  \sup_{x\in \mathsf{X}}\|\mathcal G_l(x)-\mathcal G_{l-1}(x)\|_2\le c_{\mathcal G}h_l,\qquad
  \sup_{x\in \mathsf{X}}\|\mathcal G_l(x)-\mathcal G(x)\|_2\le c_{\mathcal G}h_l ,
\end{equation}
and the first bound also holds for $\mathcal G$.
\end{enumerate}
The elliptic inverse problem of~\cite{mlsmc}, which is test case 1 of Section~\ref{sec:num-1d}, satisfies \ref{ass:A2}. There
$x$ parametrizes the coefficient $k(\cdot;x)=\bar k+\sum_jx_j\sigma_jb_j$ of the equation $-(k\,p')'=f$ on $(0,1)$ with
$p(0)=p(1)=0$, and $\mathcal G_l(x)$ collects point values of its finite-element solution on a mesh of size $h_l$. Since
$k\ge\bar k-\sum_j\sigma_j>0$ uniformly in $x$, the solution is bounded uniformly, and the finite-element error gives the
other two bounds~\cite[Section 4]{mlsmc}. The analysis below therefore covers the test case
of~\cite{mlsmc}. As for the Darcy-flow and Navier--Stokes problems of Section~\ref{sec:numerics}, their Gaussian priors are
not compactly supported and fall outside~\ref{ass:A2}, but they constitute more complicated problems on which we can test our methods. We set $\Phi^*=(R_{\mathcal G}+\|y\|_2)^2/(2\sigma_\eta^2)$ and
$c_\Phi=c_{\mathcal G}(R_{\mathcal G}+\|y\|_2)/\sigma_\eta^2$.

The following lemma shows that, under \ref{ass:A2}, consecutive posteriors are close in the 2-Wasserstein distance, at the
rate of the forward-map error. It is a property of the targets alone and involves no coupling used by the algorithm.
\begin{lem}\label{lem:elliptic}
Assume \ref{ass:A2}. Then there exists $C_0<\infty$, independent of $d$, such that for $l\ge1$ we have
\[
  W_2^2(\pi_{l-1},\pi_l)\le 2C_0e^{\Phi^*}\big(e^{2c_\Phi h_l}-1\big)^2,\qquad
  W_2^2(\pi_l,\pi)\le 2C_0e^{\Phi^*}\big(e^{2c_\Phi h_l}-1\big)^2 .
\]
\end{lem}
To turn Lemma~\ref{lem:elliptic} into a bound on the estimator, the coupling used at each level must have a cost close
to the optimal one. We formalize this as follows.
\begin{enumerate}[resume*=assumptions]
\item\label{ass:A3} For each $l\ge1$ there exist $\xi_l\in\Pi(\pi_{l-1},\pi_l)$, with conditional law $K_l$, $\kappa\ge1$ and
  $\omega_l\ge0$ such that $\E_{\xi_l}\|X-Y\|_2^2\le\kappa\,W_2^2(\pi_{l-1},\pi_l)+\omega_l$.
\end{enumerate}
Under \ref{ass:A2} and \ref{ass:A3}, Lemma~\ref{lem:elliptic} and Theorem~\ref{thm:mse} give the level variances and the bias, and hence
the multilevel rates of Corollary~\ref{cor:rates}.
\begin{cor}\label{cor:elliptic}
Assume \ref{ass:A2}--\ref{ass:A3}. Then \ref{ass:A1} holds and, for any Lipschitz $\varphi:\mathsf{X}\to\RR$ and $l\ge1$, we have
\[
  V_l\le\|\varphi\|_{\Lip}^2\big\{8\kappa C_0c_\Phi^2e^{\Phi^*+4c_\Phi h_0}h_l^2+\omega_l\big\},\qquad
  |\pi_L(\varphi)-\pi(\varphi)|\le\|\varphi\|_{\Lip}\big(2C_0e^{\Phi^*}\big)^{1/2}\big(e^{2c_\Phi h_L}-1\big).
\]
In particular, if $\omega_l\le c_\omega h_l^2$ for some $c_\omega<\infty$, the conclusions of Corollary~\ref{cor:rates} hold with
$\alpha=1$ and $\beta=2$.
\end{cor}
The optimal coupling for $W_2(\pi_{l-1},\pi_l)$, which exists since $\mathsf{X}$ is compact, satisfies \ref{ass:A3} with
$\kappa=1$ and $\omega_l=0$. The entropic coupling, which LightSB approximates, also satisfies \ref{ass:A3} with $\kappa=1$; the
next lemma bounds its excess cost $\omega_l$ in terms of the entropic parameter, for any two measures on $\mathsf X$.
\begin{lem}\label{lem:eotexcess}
Let $\mu$ and $\nu$ be probability measures on $\mathsf{X}$ with finite entropy, $\varepsilon>0$, and
$\xi_\varepsilon$ be their EOT coupling~\eqref{eq:eot-prob}. Then we have
\[
  \E_{\xi_\varepsilon}\|X-Y\|_2^2\le W_2^2(\mu,\nu)+2d\varepsilon\Big\{\log\Big(1+\frac2\varepsilon\Big)+4+2\varepsilon\Big\}.
\]
\end{lem}
For the EOT couplings of the levels, $\omega_l\le c_\omega h_l^2$ as soon as $\varepsilon_l\lesssim h_l^2/\{d\log(1/h_l)\}$: the
entropic parameter should decrease like $h_l^2$ up to a logarithmic factor. The schedule of Section~\ref{sec:numerics},
$\varepsilon_l\propto4^{-l}\propto h_l^2$, has this form, although there the coupling is learned in other coordinates.
Together, Corollary~\ref{cor:elliptic} and Lemma~\ref{lem:eotexcess} give the following cost statement. For
the one-dimensional elliptic problem of Section~\ref{sec:num-1d}, a solve on level $l$ is a tridiagonal system of cost
$O(h_l^{-1})$, so $\zeta=1$, while $\alpha=1$ and $\beta=2$; hence $\beta>\zeta$ and $\alpha\ge\frac12\min(\beta,\zeta)$. With the
entropic couplings and $\varepsilon_l\lesssim h_l^2/\{d\log(1/h_l)\}$, Corollary~\ref{cor:rates} therefore gives a mean square
error at most $\delta^2$ at cost $O(\delta^{-2})$, the cost of estimating a single expectation with independent samples
of unit cost. For a solver whose cost grows faster with the resolution, $\zeta\ge\beta$, and the cost is
$O(\delta^{-2}\log^2\delta)$ or $O(\delta^{-2-(\zeta-\beta)/\alpha})$ instead. 

This statement concerns the oracle version of
Algorithm~\ref{alg:lsb}; Corollaries~\ref{cor:fitted} and~\ref{cor:fittedmh} quantify what fitting adds.
The rates are those obtained for MLSMC in~\cite[Section 4]{mlsmc}, and a bound of order $h_l^2$ in~\eqref{eq:fe} would
give $\beta=4$ in the same way. The constant $e^{\Phi^*}$ grows as $\sigma_\eta$ decreases, as do the constants
of~\cite{mlsmc}, so these results explain the rates rather than the gain of TransportMH at small noise. The observable
$\varphi(x)=p(\frac12;x)$ of Section~\ref{sec:num-1d} is Lipschitz on $\mathsf{X}$, because $k(\cdot;x)$ is affine in $x$
and $\|k(\cdot;x)-k(\cdot;x')\|_\infty\le(\sum_j\sigma_j^2)^{1/2}\|x-x'\|_2$. Finally, an exact coupling does not depend
on the coordinates in which the particles are represented, so the bounds apply to the paired estimator whether the
samplers work in $x$ or, as in Section~\ref{sec:num-1d}, in $z=\operatorname{atanh}(x)$.

The fitted kernels do not give an exact coupling, but the same argument bounds their effect. We consider Algorithm~\ref{alg:lsb}, and then
Algorithm~\ref{alg:tmh}, under the following assumption on line~3, which separates the fit from the particles it is
applied to.
\begin{enumerate}[resume*=assumptions]
\item\label{ass:A4} For each $l\ge1$, the parameter $\theta_l$ of $K_{\theta_l}$ is fitted on samples independent of the particles.
\end{enumerate}
With $\xi^\star_l$, of conditional law $K^\star_l$, the coupling of \ref{ass:A3}, write
\[
  \mathcal E_l=\mathrm{KL}\big(\xi^\star_l\,\|\,\pi_{l-1}\otimes K_{\theta_l}\big)
  =\int\mathrm{KL}\big(K^\star_l(x,\cdot)\,\|\,K_{\theta_l}(x,\cdot)\big)\,\pi_{l-1}(dx),\qquad
  \Delta_l=\sum_{k=1}^l\mathcal E_k ,
\]
for the fitting error of level $l$ and its accumulation; for the EOT coupling, $\mathcal E_l$ is the excess of the LightSB
loss over its minimum~\cite{koro1}. The next corollary bounds the mean square error of Algorithm~\ref{alg:lsb}, which has no
Metropolis--Hastings steps, in terms of these fitting errors.
\begin{cor}\label{cor:fitted}
Assume \ref{ass:A2}--\ref{ass:A4}. Then for any Lipschitz $\varphi:\mathsf{X}\to\RR$ we have
\[
  \E\big[\{\widehat\pi_L(\varphi)-\pi(\varphi)\}^2\big]\le
  \Big\{|\pi_L(\varphi)-\pi(\varphi)|+\|\varphi\|_{\Lip}(2d\Delta_L)^{1/2}\Big\}^2
  +\bigg(\sum_{l=0}^L\frac{\tilde s_l}{\sqrt{N_l}}\bigg)^2,
\]
with $\tilde s_0=\Var_{\pi_0}(\varphi)^{1/2}$ and
$\tilde s_l^{\,2}=\|\varphi\|_{\Lip}^2\big\{\kappa W_2^2(\pi_{l-1},\pi_l)+\omega_l+2d(2\Delta_l)^{1/2}\big\}$ for $l\ge1$.
\end{cor}
Compared with Corollary~\ref{cor:elliptic}, fitting adds a bias of order $(d\Delta_L)^{1/2}$ and a term of order
$d\Delta_l^{1/2}$ to each level variance. Since $\Delta_l\ge\mathcal E_1$ does not decrease with $l$, keeping $\beta=2$ requires
$\Delta_L\lesssim h_L^4/d^2$: every fit, including the coarse ones, must be accurate at the scale of the finest level. The
bound is crude, since it charges the mismatched mass at the diameter of $\mathsf{X}$, but it shows the mechanism seen in
Section~\ref{sec:num-results}: without correction, the fitting errors accumulate into a bias that does not decrease
with $L$, and LightSB stalls.

The Metropolis--Hastings steps in line~6 of Algorithm~\ref{alg:tmh}, which leave each $\pi_l$ invariant, are what removes
this accumulation in practice. To show this, we assume that they contract in total variation.
\begin{enumerate}[resume*=assumptions]
\item\label{ass:A5} For each $l\ge1$, the transported particles are followed by $n_l$ steps of a Markov kernel $M_l$ that leaves
  $\pi_l$ invariant and has Dobrushin coefficient
  $\tau_l=\sup_{x,x'\in \mathsf{X}}\|M_l(x,\cdot)-M_l(x',\cdot)\|_{\mathrm{TV}}<1$.
\end{enumerate}
Condition \ref{ass:A5} holds, for instance, under the minorization condition verified for a blockwise Gibbs kernel
in~\cite[Proposition 4.3]{mlsmc}. Write $\pi^\theta_l$ for the law of the particles at level $l$,
$T_l=\|\pi^\theta_l-\pi_l\|_{\mathrm{TV}}$, with $T_0=0$, and $\varsigma_l=\E\|X_l^i-\widetilde X_l^i\|_2^2$ for the mean squared
displacement of the Metropolis--Hastings steps. The next corollary shows that each level now contracts the error it
inherits, so that the bias stays bounded uniformly in the number of levels.
\begin{cor}\label{cor:fittedmh}
Assume \ref{ass:A2}--\ref{ass:A5}. Then for $l\ge1$ we have
\[
  T_l\le\tau_l^{n_l}\big\{T_{l-1}+(\mathcal E_l/2)^{1/2}\big\},
\]
and, for any Lipschitz $\varphi:\mathsf{X}\to\RR$,
\[
  |\E\widehat\pi_L(\varphi)-\pi_L(\varphi)|\le2d^{1/2}\|\varphi\|_{\Lip}\,T_L,\qquad
  \Var(\widehat\pi_L(\varphi))\le\bigg(\sum_{l=0}^L\frac{\tilde s_l}{\sqrt{N_l}}\bigg)^2,
\]
with $\tilde s_0=\Var_{\pi_0}(\varphi)^{1/2}$ and
$\tilde s_l^{\,2}=2\|\varphi\|_{\Lip}^2\big\{\kappa W_2^2(\pi_{l-1},\pi_l)+\omega_l+4d\big(T_{l-1}+(\mathcal E_l/2)^{1/2}\big)+\varsigma_l\big\}$.
Moreover, if $\tau_l^{n_l}\le\bar\tau<1$ and $\mathcal E_l\le\bar{\mathcal E}$ for all $l$, then
$T_L\le\bar\tau(1-\bar\tau)^{-1}(\bar{\mathcal E}/2)^{1/2}$ for every $L$.
\end{cor}
The price of the contraction is the displacement $\varsigma_l$, which loosens the pairing and enters the level variances: more
steps reduce the bias through $\tau_l^{n_l}$ but increase $\varsigma_l$. This is the trade-off observed in
Section~\ref{sec:num-results}, where at fine levels, with the weights nearly uniform, the three sweeps multiply the
variance of the paired differences by about four.

\section{Numerical Experiments}\label{sec:numerics}
Section~\ref{sec:analysis} covers exact couplings and, through crude bounds, the fitted kernels with and without the
Metropolis--Hastings correction.
We now assess them numerically, comparing TransportMH with the multilevel sequential Monte Carlo (MLSMC) sampler
of~\cite{mlsmc} on three Bayesian inverse problems: the one-dimensional elliptic problem of~\cite{mlsmc}, with a smaller observation noise than in that
paper; a two-dimensional Darcy-flow problem; and, in the spirit of~\cite{kbj}, the recovery of the initial vorticity of
a two-dimensional Navier--Stokes flow from noisy point observations at later times. The first two have finite-element
forward maps; the third is nonlinear and time dependent. In all three the first transition is hard and the later ones
are easy: the ESS fraction of $\gamma_1/\gamma_0$ on a $\pi_0$ cloud is 3\%, 1.2\% and 1.8\%, whereas at the second
transition it is already about 60\% and 73\% in test cases 2 and 3. All computations are in Python (NumPy/SciPy for the
PDE solves, PyTorch for the LightSB fits) on a single CPU thread of an Apple M4 Pro laptop. Implementation details,
the sample sizes, the computation of the reference values and a breakdown of the computing time are in
Appendix~\ref{app:impl}.

\subsection{Common setup}\label{sec:num-setup}

\paragraph{Targets.} For an unknown parameter $x$ with prior $\pi_{\mathrm{pr}}$, data $y\in\RR^{d_y}$ and a forward
map $\mathcal G_l$ computed on a mesh of size $h_l$ (finite elements in test cases 1 and 2, finite differences in space
and time in test case 3), level $l$ has unnormalized density
\[
  \gamma_l(x)=\exp\Big(-\frac1{2\sigma_\eta^2}\big\|y-\mathcal G_l(x)\big\|_2^2\Big)\,\pi_{\mathrm{pr}}(x),
  \qquad \pi_l(d x)=\frac{\gamma_l(x)d x}{Z_l},
\]
and the incremental weights are $\gamma_l/\gamma_{l-1}$, the ratio of the level-$l$ and level-$(l-1)$ likelihoods.
All samplers work in unconstrained coordinates $z\in\RR^d$, a bijective reparametrization of $x$ defined for each
problem below, and $\pi_l$ also denotes the target in these coordinates. The quantity of interest is $\pi_L(\varphi)$ for an observable $\varphi$ given for each problem.

\paragraph{Methods.} All four methods start from the same $N_0$ samples of $\pi_0$ and use the same $N_0\ge\dots\ge N_L$.
\begin{enumerate}[label=(\roman*)]
\item \emph{MLSMC-pCN}: the sampler of~\cite{mlsmc}: reweight by $\gamma_l/\gamma_{l-1}$, resample to $N_l$ particles,
  and mutate with 3 sweeps of the pCN-MH kernel of Section~\ref{sec:tmh}, with $(m,AA^\top)$ the mean and covariance of
  the current cloud and $\rho=0.78$.
\item \emph{MLSMC-pCN + ESS sweeps}: as (i), with $n_l$ sweeps chosen by the rule of Section~\ref{sec:tmh}, so that
  (ii) and (iv) differ only by the transport.
\item \emph{LightSB}: Algorithm~\ref{alg:lsb}, without correction.
\item \emph{TransportMH}: Algorithm~\ref{alg:tmh}, with the pCN-MH kernel and the adaptive $n_l$ and $\rho_l$ of
  Section~\ref{sec:tmh}.
\end{enumerate}
The mutations of (i) and (ii) and the MH step of (iv) leave $\pi_l$ invariant; (iii) and (iv) use the paired estimator
\eqref{eq:estimator}. TransportMH evaluates $\gamma_l$ at the transported particles and runs its MH sweeps at
every level, including $L$, where MLSMC does not mutate; at equal $L$ it is therefore the most expensive method
(Appendix~\ref{app:solves}), and it can only win at equal accuracy.

\paragraph{LightSB fits.} The potentials $v_\theta$ have four components with diagonal covariances and are fitted with
Adam on the whitened level-$(l-1)$ cloud (Appendix~\ref{app:fits}). The entropic parameter decreases geometrically,
$\varepsilon_l=0.05\cdot4^{-(l-1)}$, floored at $10^{-5}$, following the shrinking distance between consecutive levels.

\paragraph{Sample sizes.} We choose the sample sizes as Corollary~\ref{cor:rates} advises, as
in~\cite{mlsmc}:
\[
  N_l\propto\delta^{-2}\left(\sum_{k\le L}h_k^{(\beta-\zeta)/2}\right)h_l^{(\beta+\zeta)/2}
\]
for a tolerance $\delta$, floored at $N_l\ge100$ so that the cloud covariance up to $d=50$ is well defined. With the rates of
Table~\ref{tab:settings}, $\beta>\zeta$ in all three problems, so this allocation targets a mean square error $\delta^2$ at
cost $O(\delta^{-2})$, the regime of the cost statement of Section~\ref{sec:elliptic-theory}; increasing the horizon $L$
by one halves $\delta$, as for $\alpha=1$, which multiplies $N_0$ by about four. Table~\ref{tab:settings}
gives the rates used and Appendix~\ref{app:settings} the resulting $N_l$.

\paragraph{Error and cost.} Each configuration is run $R=30$ times independently. With reference value $\bar\pi$ and
estimates $\widehat \pi_{L,r}(\varphi)$,
\[
  \mathrm{relRMSE}=\frac{1}{|\bar\pi|}\Big\{\frac1R\sum_{r=1}^R\big(\widehat \pi_{L,r}(\varphi)-\bar\pi\big)^2\Big\}^{1/2},\qquad
  \mathrm{relbias}=\frac{1}{|\bar\pi|}\Big(\frac1R\sum_{r=1}^R\widehat \pi_{L,r}(\varphi)-\bar\pi\Big),
\]
and the standard error of the relative bias combines the spread of the runs and the standard error of $\bar\pi$. Cost
is the mean wall time per run, including LightSB fitting and all PDE solves and excluding the generation of the
level-0 samples, which is shared by all methods. Varying the horizon $L$ traces an error--cost curve per method.

\subsection{Test case 1: one-dimensional elliptic problem}\label{sec:num-1d}

This is the example of~\cite{mlsmc}, in our own implementation.
\paragraph{Model.} On $\mathcal D=(0,1)$,
\[
  -\frac{d}{d s}\Big(k(s;x)\frac{d p}{d s}(s;x)\Big)=100\,s,\qquad p(0;x)=p(1;x)=0,
\]
with
\[
  k(s;x)=0.15+\sum_{j=1}^{50}x_j\,\sigma_j\,b_j(s),\qquad \sigma_j=0.4\cdot4^{-j},\qquad
  b_j(s)=\begin{cases}\sin(j\pi s),& j\text{ odd},\\ \cos(j\pi s),& j\text{ even},\end{cases}
\]
and prior $x\sim\mathcal U([-1,1]^{50})$, so that $k\ge0.15-\sum_j\sigma_j>0$. The samplers work in
$z=\operatorname{atanh}(x)\in\RR^{50}$, where the prior density is $\prod_j\operatorname{sech}^2z_j$.

\paragraph{Discretization and data.} Level $l$ uses continuous piecewise-linear finite elements on $2^{l+3}$ uniform
elements, $h_l=2^{-(l+3)}$; one solve is a tridiagonal system, of cost $O(h_l^{-1})$. The data are
$y=(p(\frac14;x^\dagger),p(\frac34;x^\dagger))+\eta$, $\eta\sim\Na(0,\sigma_\eta^2I_2)$, with $x^\dagger\sim\mathcal U([-1,1]^{50})$
(fixed seed), computed on $2^{20}$ elements, and $\sigma_\eta=0.05$ (\cite{mlsmc} use $0.25$; the smaller noise makes
consecutive levels further apart). The observable is $\varphi(x)=p(\frac12;x)$, computed from the exact solution
of the continuum problem, the same at every level.

\subsection{Test case 2: two-dimensional Darcy flow}\label{sec:num-darcy}

\paragraph{Model.} On $\mathcal D=(0,1)^2$,
\[
  -\nabla\cdot\big(k(s;x)\nabla p(s;x)\big)=0,\qquad p=1\ \text{on }s_1=0,\quad p=0\ \text{on }s_1=1,\quad
  k\nabla p\cdot n=0\ \text{on }s_2\in\{0,1\},
\]
with log-permeability
\[
  \log k(s;x)=\sum_{j=1}^{25}x_j\,\lambda_j^{1/2}e_j(s),\qquad x\sim\Na(0,I_{25}),
\]
where the $e_j$ are the $L^2$-normalized cosine modes $c_mc_n\cos(m\pi s_1)\cos(n\pi s_2)$, $(m,n)\neq(0,0)$, with
eigenvalues $\lambda\propto\{\pi^2(m^2+n^2)+9\}^{-2}$, the 25 largest modes are kept, and the $\lambda_j$ are scaled so that the
average pointwise variance of $\log k$ is 1. The constant mode is excluded because $p$ is invariant under
$k\mapsto ck$. The samplers work in $z=x\in\RR^{25}$.

\paragraph{Discretization and data.} Level $l$ uses P1 finite elements on a uniform mesh of $(4\cdot2^l)\times(4\cdot2^l)$
squares, each cut into two triangles. Pressure is observed at the $5\times5$ points $(i/6,j/6)$, $i,j=1,\dots,5$, with
$\sigma_\eta=0.025$, from $x^\dagger\sim\Na(0,I_{25})$ (fixed seed) on level 6; sensor values converge at rate $h_l^2$.
The observable is the mean permeability over the center of the domain, $\varphi(x)=4\int_{[1/4,3/4]^2}k(s;x)\,ds$, which
needs no PDE solve. It is not globally Lipschitz in $x$, so Section~\ref{sec:analysis} does not cover it.

\subsection{Test case 3: two-dimensional Navier--Stokes}\label{sec:num-ns}
\paragraph{Model.} On the unit torus $\mathbb T^2=[0,1)^2$, the vorticity $w(s,t)$ of an incompressible flow satisfies
\[
  \partial_tw+u\cdot\nabla w=\nu\Delta w+f,\qquad u=(\partial_2\Psi,-\partial_1\Psi),\qquad -\Delta\Psi=w,
  \qquad t\in(0,1],
\]
with viscosity $\nu=0.01$ and fixed forcing $f(s)=\frac12\{\sin2\pi(s_1+s_2)+\cos2\pi(s_1+s_2)\}$. The unknown is the
initial vorticity
\[
  w(s,0;x)=\sum_{k=1}^{36}x_k\,\lambda_k^{1/2}e_k(s),\qquad x\sim\Na(0,I_{36}),
\]
where the $e_k$ are the $L^2$-normalized real Fourier modes $\sqrt2\cos(2\pi q\cdot s)$, $\sqrt2\sin(2\pi q\cdot s)$
for the 18 wavevectors $q\in\mathbb Z^2\setminus\{0\}$ of smallest $|q|$ (one from each pair $\pm q$),
$\lambda_k\propto(4\pi^2|q|^2+9)^{-2}$, and the $\lambda_k$ are scaled so that the average pointwise prior standard
deviation of $w(\cdot,0)$ is 2. The samplers work in $z=x\in\RR^{36}$.

\paragraph{Discretization and data.} Level $l$ uses an $(8\cdot2^l)\times(8\cdot2^l)$ grid with second-order central
differences in space and the classical fourth-order Runge--Kutta method with an integrating factor for the viscous term,
with $8\cdot2^l$ time steps. Sensor values converge at rate $h_l^2$ and the cost of a solve grows like $h_l^{-3}$ up to
a logarithm. We use finite differences rather than a pseudo-spectral discretization on purpose: for this smooth flow
the spectral solution converges exponentially and the level corrections vanish once the grid resolves the flow, which
leaves nothing for a multilevel method to do. The vorticity is observed at the $4\times4$ points $(i/4,j/4)$,
$i,j=0,\dots,3$, at times $t=0.5$ and $t=1$, giving 32 observations with $\sigma_\eta=0.25$, from
$x^\dagger\sim\Na(0,I_{36})$ (fixed seed) on level 5. The observable is the mean initial vorticity over a quarter of the
domain, $\varphi(x)=4\int_{[0,1/2]^2}w(s,0;x)\,ds$, which needs no PDE solve and is linear in $x$, hence Lipschitz.

\begin{table}[!htbp]
\centering\footnotesize
\begin{tabular}{lccccll}
\toprule
problem & $d$ & $(\alpha,\beta,\zeta)$ & horizons & $N_0$ at finest $L$ & target & reference (SE)\\
\midrule
1D elliptic & 50 & $(1,2,1)$ & $L=1,\dots,9$ & $10^5$ & $\pi(\varphi)$ & $50.33335\ (0.00010)$\\
Darcy flow & 25 & $(\cdot,4,2.5)$ & $L=1,\dots,6$ & $8\times10^4$ & $\pi_5(\varphi)$ & $0.99780\ (0.00112)$\\
Navier--Stokes & 36 & $(\cdot,4,3)$ & $L=1,\dots,4$ & $2\times10^4$ & $\pi_3(\varphi)$ & $0.7981\ (0.0029)$\\
\bottomrule
\end{tabular}
\caption{Settings of the three test cases: dimension, rates used for the sample sizes, horizons, level-0 sample size at
the finest horizon, target of the scored estimates and reference value with its standard error. For the Darcy-flow and Navier--Stokes problems
the finer levels change the target by much less than the reference standard error. Appendix~\ref{app:settings} gives the
full allocations and how the reference values were computed.}\label{tab:settings}
\end{table}

\subsection{Results}\label{sec:num-results}

\paragraph{Error against cost.} Figure~\ref{fig:err-time} and Tables~\ref{tab:results}--\ref{tab:time-to-target} give
the main comparison. In test cases 1 and 2 the MLSMC variants are cheaper at loose accuracy, where the LightSB fit is a
fixed overhead, and TransportMH is faster at tight accuracy: it reaches a relative RMSE of $3\times10^{-4}$ in 3.5\,s
against 5.9\,s in test case 1, and $10^{-2}$ in 13\,s against 40\,s in test case 2, and it is the only method to reach
the smallest targets. In test case 3 it is faster at every accuracy ($2\times10^{-2}$ in 5.0\,s against 59\,s), and at
$L=4$ its mean squared error is 4.5 times smaller than that of MLSMC-pCN with ESS sweeps, for 1.8 times the wall time;
plain MLSMC-pCN hardly improves beyond $L=2$. LightSB without correction stalls at a relative bias of order
$10^{-3}$--$10^{-2}$, several standard errors from zero: the transport error accumulates over levels, and the MH step
is what removes it.

\begin{figure}[htbp]
\centering
\includegraphics[width=\textwidth]{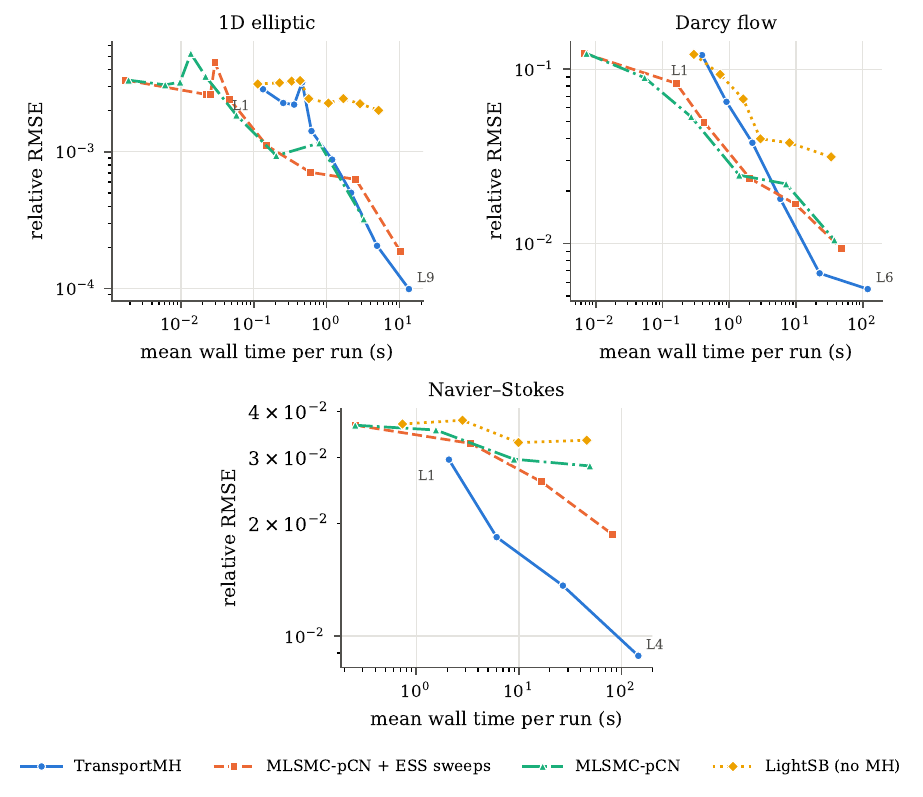}
\caption{Relative RMSE against mean wall time per run, over horizons $L$ (30 runs per point): test case 1,
$L=1,\dots,9$; test case 2, $L=1,\dots,6$; test case 3, $L=1,\dots,4$. Labels on the TransportMH curve give the
smallest and largest $L$.}\label{fig:err-time}
\end{figure}

\begin{figure}[htbp]
\centering
\includegraphics[width=\textwidth]{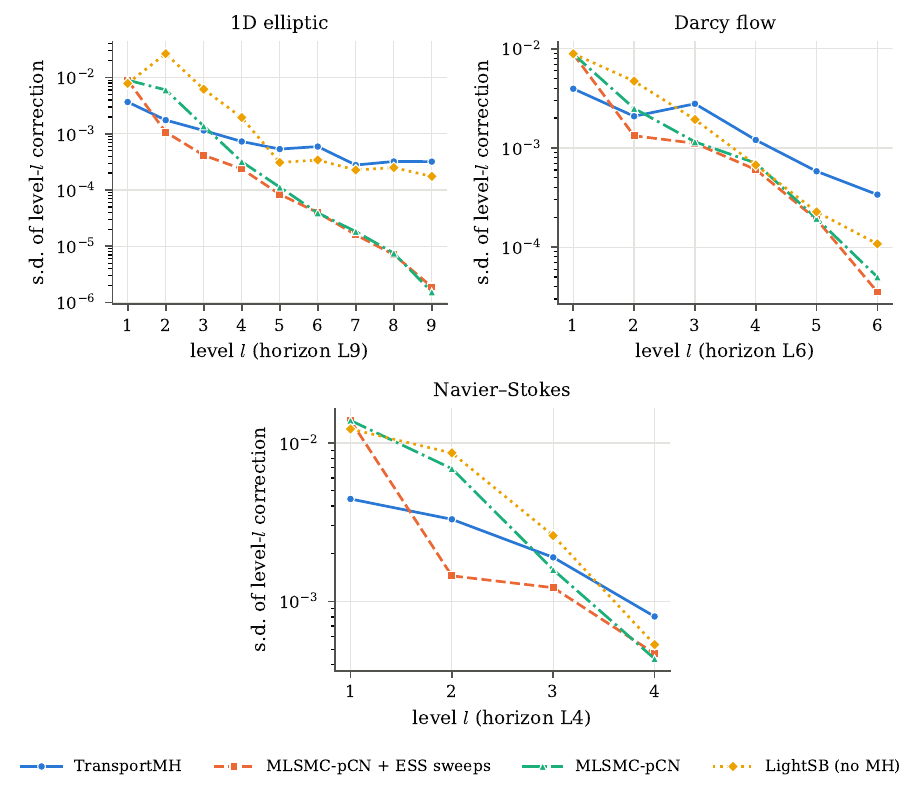}
\caption{Standard deviation over the 30 runs of the level-$l$ correction at the finest horizon. For the paired
methods the level-$l$ correction is $N_l^{-1}\sum_{i\le N_l}\{\varphi(X_l^i)-\varphi(X_{l-1}^i)\}$; for MLSMC it is the
$l$th increment of~\cite{mlsmc}. Up to covariances between levels, which are small here, the variance of the estimator is the sum of the squared values plotted; TransportMH is lowest at level 1, which dominates for all methods.}\label{fig:level-sd}
\end{figure}

\begin{table}[!htbp]
\centering\footnotesize
\setlength{\tabcolsep}{3pt}
\begin{tabular}{l crrr crrr crrr}
\toprule
& \multicolumn{4}{c}{1D elliptic ($\times10^{-4}$)} & \multicolumn{4}{c}{Darcy flow ($\times10^{-3}$)} & \multicolumn{4}{c}{Navier--Stokes ($\times10^{-3}$)}\\
\cmidrule(lr){2-5}\cmidrule(lr){6-9}\cmidrule(lr){10-13}
method & $L$ & RMSE & bias (SE) & time & $L$ & RMSE & bias (SE) & time & $L$ & RMSE & bias (SE) & time\\
\midrule
MLSMC-pCN & 7 & 9.4 & $-1.5\,(1.7)$ & 0.2 & 4 & 25 & $-8.6\,(4.4)$ & 1.4 & 2 & 36 & $+7.8\,(7.4)$ & 1.6\\
 & 8 & 12 & $+2.2\,(2.1)$ & 0.8 & 5 & 22 & $-5.4\,(4.1)$ & 7.1 & 3 & 30 & $-4.5\,(6.6)$ & 9.0\\
 & 9 & 3.2 & $-0.70\,(0.58)$ & 3.2 & 6 & 10 & $+0.1\,(2.2)$ & 37.6 & 4 & 29 & $-10\,(6)$ & 49.1\\
\addlinespace
+ ESS sweeps & 7 & 7.1 & $-3.0\,(1.2)$ & 0.6 & 4 & 24 & $-9.7\,(4.2)$ & 2.0 & 2 & 33 & $+6.7\,(7.0)$ & 3.4\\
 & 8 & 6.3 & $+1.0\,(1.2)$ & 2.5 & 5 & 17 & $-1.8\,(3.3)$ & 9.7 & 3 & 26 & $-5.0\,(5.9)$ & 16.6\\
 & 9 & 1.9 & $-0.54\,(0.33)$ & 10.2 & 6 & 9.3 & $+0.6\,(2.1)$ & 48.6 & 4 & 19 & $-8.3\,(4.8)$ & 81.0\\
\addlinespace
LightSB & 7 & 24 & $-23\,(2)$ & 1.7 & 4 & 40 & $-31\,(5)$ & 2.9 & 2 & 38 & $+7.0\,(7.8)$ & 2.8\\
 & 8 & 22 & $-21\,(1)$ & 2.9 & 5 & 38 & $-28\,(5)$ & 8.0 & 3 & 33 & $-12\,(7)$ & 9.9\\
 & 9 & 20 & $-19\,(1)$ & 5.2 & 6 & 31 & $-29\,(2)$ & 33.6 & 4 & 33 & $-20\,(6)$ & 45.7\\
\addlinespace
TransportMH & 7 & 5.0 & $-3.3\,(0.7)$ & 2.2 & 4 & 18 & $-3.3\,(3.5)$ & 5.8 & 2 & 18 & $-4.2\,(4.9)$ & 6.1\\
 & 8 & 2.1 & $-0.87\,(0.35)$ & 4.9 & 5 & 6.8 & $-1.6\,(1.7)$ & 22.6 & 3 & 14 & $-4.0\,(4.4)$ & 26.8\\
 & 9 & $\mathbf{0.99}$ & $-0.47\,(0.16)$ & 13.5 & 6 & $\mathbf{5.5}$ & $-3.4\,(1.4)$ & 118.9 & 4 & $\mathbf{8.8}$ & $-5.6\,(3.8)$ & 145.7\\
\bottomrule
\end{tabular}
\caption{Relative RMSE and relative bias with its standard error, in the units given in each column header, and mean
wall time per run in seconds, at the three finest horizons $L$; 30 runs each.
``+ ESS sweeps'' is MLSMC-pCN with ESS-adaptive sweeps.}\label{tab:results}
\end{table}

\begin{table}[!htbp]
\centering\footnotesize
\setlength{\tabcolsep}{5pt}
\begin{tabular}{l cccc cccc cccc}
\toprule
& \multicolumn{4}{c}{1D elliptic ($\times10^{-4}$)} & \multicolumn{4}{c}{Darcy flow ($\times10^{-3}$)} & \multicolumn{4}{c}{Navier--Stokes ($\times10^{-3}$)}\\
\cmidrule(lr){2-5}\cmidrule(lr){6-9}\cmidrule(lr){10-13}
relRMSE target & 10 & 3 & 1.5 & 1 & 30 & 20 & 10 & 6.6 & 30 & 20 & 15 & 10\\
\midrule
MLSMC-pCN & 0.2 & -- & -- & -- & 0.9 & 8.8 & -- & -- & 8.3 & -- & -- & --\\
+ ESS sweeps & 0.2 & 5.9 & -- & -- & 1.2 & 4.4 & 40 & -- & 6.1 & 59 & -- & --\\
TransportMH & 1.0 & \textbf{3.5} & \textbf{7.6} & \textbf{13.3} & 3.0 & 5.1 & \textbf{13} & \textbf{27} & \textbf{2.1} & \textbf{5.0} & \textbf{17} & \textbf{90}\\
\bottomrule
\end{tabular}
\caption{Wall time (seconds) to reach a relative RMSE target, given in the units of each column header, by log--log interpolation between consecutive horizons; ``--'':
not reached at the finest horizon. LightSB reaches none of these targets.}\label{tab:time-to-target}
\end{table}

\paragraph{Where the gain comes from.} Figure~\ref{fig:level-sd} shows that the level-1 correction dominates the
variance for all methods, and that TransportMH reduces its variance by factors of about 6, 5 and 10 in test cases 1--3
(standard deviations 0.0037, 0.0040 and 0.0044 against 0.0090, 0.0089 and 0.0139 for MLSMC-pCN). In MLSMC the level-1
increment is an importance-sampling average over the $N_0$ particles at an ESS of a few percent, computed before the
mutation, so extra sweeps cannot reduce it: the level-1 curves of the two MLSMC variants coincide, and the sweeps only
help from level 2 on, where MLSMC-pCN with ESS sweeps has the smaller corrections. In TransportMH the level-1 correction
is a paired difference along the transport, whose variance is controlled by the transport cost rather than by the ESS.

\paragraph{Computing time.} TransportMH and MLSMC-pCN with ESS sweeps make almost the same number of PDE solves, most
of them in the 30 sweeps at the first transition, but the extra solves of TransportMH, at the transported particles and
in the MH sweeps at the final levels, sit where solves are most expensive (Appendix~\ref{app:cost}). The LightSB fits
cost about 3\,s per run whatever the $N_l$: a fifth of the total in test case 1, where solves are cheap, and 2--2.5\% in
test cases 2 and 3. At fine levels, where $e_l\approx1$, the MH sweeps move little (acceptance 94--99\%) but each adds
about as much noise to $X_l^i-X_{l-1}^i$ as the transport, so three sweeps multiply the variance of the paired
differences by about four. Skipping them there would bring the $L=6$ run of test case 2 from 119\,s to roughly
60--65\,s; we have not tested its effect on the bias.

\paragraph{Limitations.} Two features of TransportMH call for comment.
\begin{enumerate}[label=(\alph*)]
\item \emph{Fine-level variances.} The level-$l$ correction standard deviation of TransportMH levels off at about
  $3\times10^{-4}$ in test case 1 (levels 7--9), while those of the MLSMC variants keep decreasing, to $2\times10^{-6}$ at
  level 9, because their fine-level increments are importance-sampling corrections with weights close to one. The
  plateau comes from two implementation choices rather than from the transport: the floor $\varepsilon_l\ge10^{-5}$,
  reached at level 7, stops the transport noise from decreasing (the variance of the paired differences before MH is
  $3.0$, $2.6$ and $2.5\times10^{-6}$ at levels 7--9), and the MH sweeps multiply it by about four; with $N_l=100$,
  $(9.7\times10^{-6}/100)^{1/2}\approx3\times10^{-4}$. In test case 2, where the floor is not reached, the correction
  standard deviation roughly halves per level at fixed $N_l=100$ ($1.2\times10^{-3}$, $5.8\times10^{-4}$ and
  $3.4\times10^{-4}$ at levels 4--6), the decay $V_l\lesssim h_l^\beta$ of Corollary~\ref{cor:rates} with $\beta=2$ from
  $\varepsilon_l\propto h_l^2$. The gain of TransportMH is nevertheless concentrated at the coarse, hard transitions.
\item \emph{Bias at the finest horizon.} TransportMH has a small negative relative bias at the finest horizon:
  $(-4.7\pm1.6)\times10^{-5}$ in test case 1 and $(-3.4\pm1.4)\times10^{-3}$ in test case 2, about 3 and 2.4 standard
  errors, and $(-5.6\pm3.8)\times10^{-3}$ in test case 3, where the standard error is dominated by the reference value.
  This is consistent with an inexact transport followed by finitely many MH sweeps, which reduce but do not remove the
  transport bias.
\end{enumerate}
On problems where all transitions have a high ESS (not shown) the MLSMC variants are faster at every accuracy we
tested: the advantage of the transport is specific to transitions where the incremental weights degenerate.

\section{Discussion}\label{sec:discussion}

Learned entropic transport between consecutive posteriors, corrected by Metropolis--Hastings steps that keep the
pairing, helps most at the coarse transitions where the incremental weights degenerate; at fine levels the
importance-sampling corrections of MLSMC are already accurate. This suggests a hybrid scheme, with transport only where
the effective sample size is low and fewer or no fine-level sweeps, together with an adaptive choice of $\varepsilon_l$
and of the number of mixture components. On the theoretical side, the analysis covers exact couplings and bounds the effect of the fitted kernels only crudely; a sharper
treatment requires bounds on the error of the stochastic gradient fit, on the approximation of the oracle potential by
Gaussian mixtures, and on the contraction of the preconditioned Crank--Nicolson kernel used in practice. Finally, higher-dimensional parameters, such as discretized random fields,
would call for more structured potentials, for instance with low-rank covariances.

\appendix
\section{Proofs for Section~\ref{sec:analysis}}\label{app:proofs}

\begin{proof}[Proof of Corollary~\ref{cor:rates}]
The proof has two steps: the first bound of Theorem~\ref{thm:mse} is summed with the assumed rate of the
level variances, and the result is fed to the standard multilevel complexity argument.

(i) \emph{Bias and variance.} Under \ref{ass:A1}, $\E\widehat\pi_L(\varphi)=\pi_L(\varphi)$ and
the mean square error is the squared discretization bias plus the variance:
\[
  \E\big[\{\widehat \pi_L(\varphi)-\pi(\varphi)\}^2\big]=\{\pi_L(\varphi)-\pi(\varphi)\}^2+\Var\big(\widehat\pi_L(\varphi)\big)
  \le c_1^2h_L^{2\alpha}+\Var\big(\widehat\pi_L(\varphi)\big).
\]

(ii) \emph{Summing the covariances.} The proof of Theorem~\ref{thm:mse} gives
$\Var(\widehat\pi_L(\varphi))\le\sum_lV_l/N_l+2\sum_{k<l}\sqrt{V_kV_l}/N_k$. Grouping the second sum by $k$,
\[
  \Var\big(\widehat\pi_L(\varphi)\big)\le\sum_{l=0}^L\frac{V_l}{N_l}
  +2\sum_{k=0}^{L-1}\frac{\sqrt{V_k}}{N_k}\sum_{l=k+1}^L\sqrt{V_l}.
\]
Since $h_l=h_k2^{-(l-k)}$, the inner sum is a geometric series: by assumption,
\[
  \sum_{l=k+1}^L\sqrt{V_l}\le\sqrt{c_2}\,h_k^{\beta/2}\sum_{m\ge1}2^{-m\beta/2}=c_\beta\sqrt{c_2}\,h_k^{\beta/2}.
\]
For $k=0$ this contributes $2c_\beta\sqrt{V_0c_2}\,h_0^{\beta/2}/N_0$. For $k\ge1$, the assumption also gives
$\sqrt{V_k}\le\sqrt{c_2}\,h_k^{\beta/2}$, so the contribution is at most $2c_\beta c_2h_k^\beta/N_k$; adding
$V_k/N_k\le c_2h_k^\beta/N_k$ gives $(1+2c_\beta)c_2h_k^\beta/N_k$. Together with (i) this is the displayed bound.

(iii) \emph{Complexity.} The displayed bound has the form
\[
  \E\big[\{\widehat \pi_L(\varphi)-\pi(\varphi)\}^2\big]\le c_1^2h_L^{2\alpha}+\sum_{l=0}^L\frac{v_l}{N_l},
  \qquad v_l\le c_2'h_l^\beta,
\]
with $c_2'=\max\{(1+2c_\beta)c_2,\ h_0^{-\beta}(V_0+2c_\beta\sqrt{V_0c_2}\,h_0^{\beta/2})\}$. This is exactly what the proof of the
multilevel complexity theorem of \cite{cgst} (Theorem~1; see \cite{giles} for the original result) uses: there,
independence of the level corrections serves only to obtain a variance of the form $\sum_lv_l/N_l$, and the remaining
inputs are the bias bound and the cost bound, both assumed here. That proof takes $L$ as the smallest level with
$c_1^2h_L^{2\alpha}\le\delta^2/2$ and $N_l=\lceil2\,\delta^{-2}c_2'S_L\,h_l^{(\beta+\zeta)/2}\rceil$ with
$S_L=\sum_{k=0}^Lh_k^{(\beta-\zeta)/2}$, so that the variance is at most $\delta^2/2$, and the cost
$c_3\sum_lN_lh_l^{-\zeta}$ then gives the three regimes. These $N_l$ are non-increasing in $l$, as Theorem~\ref{thm:mse}
requires, because $h_l$ decreases.
\end{proof}

\begin{proof}[Proof of Lemma~\ref{lem:elliptic}]
We say that a probability measure $\mu$ on $\RR^d$ satisfies a logarithmic Sobolev inequality with constant $C$ if
$\mathrm{Ent}_\mu(f^2)\le2C\int\|\nabla f\|_2^2\,d\mu$ for all smooth $f$, where
$\mathrm{Ent}_\mu(g)=\mu(g\log g)-\mu(g)\log\mu(g)$.

\emph{Density ratio.} Write $\Delta_l=\Phi_l-\Phi_{l-1}$. Since
$\Phi_l-\Phi_{l-1}=\langle\mathcal G_{l-1}-\mathcal G_l,\,2y-\mathcal G_l-\mathcal G_{l-1}\rangle/(2\sigma_\eta^2)$,
\eqref{eq:fe} and the Cauchy--Schwarz inequality give $\sup_\mathsf{X}|\Delta_l|\le c_\Phi h_l$. The density of $\pi_l$ with respect
to $\pi_{l-1}$ is $g_l=e^{-\Delta_l}/\pi_{l-1}(e^{-\Delta_l})$, a ratio of two quantities in
$[e^{-c_\Phi h_l},e^{c_\Phi h_l}]$, so $g_l\in[e^{-2c_\Phi h_l},e^{2c_\Phi h_l}]$ and $|g_l-1|\le e^{2c_\Phi h_l}-1$. Hence
\[
  \mathrm{KL}(\pi_l\|\pi_{l-1})\le\chi^2(\pi_l\|\pi_{l-1})=\int(g_l-1)^2\,d\pi_{l-1}\le\big(e^{2c_\Phi h_l}-1\big)^2,
\]
using $\mathrm{KL}\le\log(1+\chi^2)\le\chi^2$.

\emph{Transport inequality.} The uniform law on $[-1,1]$ satisfies a logarithmic Sobolev inequality with some constant
$C_0$~\cite{bgl}, and such inequalities tensorize, so $\pi_{\mathrm{pr}}$ satisfies one with the same constant, whatever $d$. The
measure $\pi_{l-1}$ has density proportional to $e^{-\Phi_{l-1}}$ with respect to $\pi_{\mathrm{pr}}$, and $0\le\Phi_{l-1}\le\Phi^*$ on
$\mathsf{X}$ by~\eqref{eq:fe}. By the perturbation lemma of~\cite{hs}, $\pi_{l-1}$ satisfies a logarithmic Sobolev inequality with
constant $C_0e^{\Phi^*}$, and by~\cite{ov} it then satisfies the transport inequality
$W_2^2(\nu,\pi_{l-1})\le2C_0e^{\Phi^*}\mathrm{KL}(\nu\|\pi_{l-1})$ for every probability measure $\nu$. With $\nu=\pi_l$ and
the bound on the Kullback--Leibler divergence, this gives the first bound. The second follows in the same way, with
$\Phi-\Phi_l$ in place of $\Delta_l$, the last bound in~\eqref{eq:fe}, and the transport inequality for $\pi$.
\end{proof}
\begin{proof}[Proof of Corollary~\ref{cor:elliptic}]
\ref{ass:A1} holds because $\xi_l\in\Pi(\pi_{l-1},\pi_l)$. As in the proof of Theorem~\ref{thm:mse},
\[
  V_l\le\|\varphi\|_{\Lip}^2\E_{\xi_l}\|X-Y\|_2^2\le\|\varphi\|_{\Lip}^2\{\kappa W_2^2(\pi_{l-1},\pi_l)+\omega_l\},
\]
and by Lemma~\ref{lem:elliptic} and $e^{2x}-1\le2xe^{2x}$, $W_2^2(\pi_{l-1},\pi_l)\le8C_0c_\Phi^2e^{\Phi^*+4c_\Phi h_0}h_l^2$.
For the bias,
\[
  |\pi_L(\varphi)-\pi(\varphi)|\le\|\varphi\|_{\Lip}W_1(\pi_L,\pi)\le\|\varphi\|_{\Lip}W_2(\pi_L,\pi),
\]
with
$W_1$ defined as $W_2$ with the cost $\|X-Y\|_2$, and Lemma~\ref{lem:elliptic} applies.
\end{proof}
\begin{proof}[Proof of Lemma~\ref{lem:eotexcess}]
\emph{Reformulation.} For $\xi\in\Pi(\mu,\nu)$, $H(\xi)=H(\mu)+H(\nu)-\mathrm{KL}(\xi\|\mu\otimes\nu)$, so $\xi_\varepsilon$ also minimizes
$\frac12\E_\xi\|X-Y\|_2^2+\varepsilon\,\mathrm{KL}(\xi\|\mu\otimes\nu)$ over $\Pi(\mu,\nu)$. Comparing with any $\tilde\xi\in\Pi(\mu,\nu)$ and
dropping $\mathrm{KL}(\xi_\varepsilon\|\mu\otimes\nu)\ge0$,
\[
  \E_{\xi_\varepsilon}\|X-Y\|_2^2\le\E_{\tilde\xi}\|X-Y\|_2^2+2\varepsilon\,\mathrm{KL}(\tilde\xi\|\mu\otimes\nu).
\]

\emph{Competitor.} Let $N_\varepsilon=\lceil2/\varepsilon\rceil$ and partition $\mathsf{X}$ into $N_\varepsilon^d$ half-open cubes $Q$ of side $r=2/N_\varepsilon\le\varepsilon$,
of diameter $r\sqrt d$. Let $\xi_0$ be an optimal coupling for $W_2(\mu,\nu)$ and $\bar\xi_{QQ'}=\xi_0(Q\times Q')$. Draw
$(X_0,Y_0)\sim\xi_0$ and, given that $(X_0,Y_0)\in Q\times Q'$, draw $X$ and $Y$ independently from $\mu$ restricted to $Q$
and $\nu$ restricted to $Q'$, normalized. The law $\tilde\xi$ of $(X,Y)$ has density $\bar\xi_{QQ'}/\{\mu(Q)\nu(Q')\}$ with respect to
$\mu\otimes\nu$ on $Q\times Q'$, and marginals $\sum_Q\mu(Q)\,\mu(\cdot\mid Q)=\mu$ and $\nu$, so $\tilde\xi\in\Pi(\mu,\nu)$.

\emph{Cost.} Since $\|X-X_0\|_2$ and $\|Y-Y_0\|_2$ are at most $r\sqrt d$, and $\E\|X_0-Y_0\|_2\le W_2(\mu,\nu)\le2\sqrt d$,
\[
  \E_{\tilde\xi}\|X-Y\|_2^2\le\E\big(\|X_0-Y_0\|_2+2r\sqrt d\big)^2\le W_2^2(\mu,\nu)+8dr+4dr^2\le W_2^2(\mu,\nu)+8d\varepsilon+4d\varepsilon^2 .
\]

\emph{Entropy.} $\mathrm{KL}(\tilde\xi\|\mu\otimes\nu)=\sum_{QQ'}\bar\xi_{QQ'}\log[\bar\xi_{QQ'}/\{\mu(Q)\nu(Q')\}]$ is the mutual information of
the cells of $X_0$ and $Y_0$, which is at most the entropy of the cell of $X_0$, hence at most
$\log N_\varepsilon^d\le d\log(1+2/\varepsilon)$. Combining the three displays gives the bound.
\end{proof}

\begin{proof}[Proof of Corollary~\ref{cor:fitted}]
Write $\pi^\theta_l$ for the law of the particles at level $l$ and $\xi^\theta_l=\pi^\theta_{l-1}\otimes K_{\theta_l}$ for the law of
$(X^i_{l-1},X^i_l)$; we condition on the fitted parameters, which are independent of the particles by \ref{ass:A4}.

\emph{Accumulation.} By the chain rule for the Kullback--Leibler divergence,
\[
  \mathrm{KL}(\xi^\star_l\|\xi^\theta_l)=\mathrm{KL}(\pi_{l-1}\|\pi^\theta_{l-1})+\mathcal E_l,
\]
and since $\pi_l$ and $\pi^\theta_l$ are the
second marginals of $\xi^\star_l$ and $\xi^\theta_l$, $\mathrm{KL}(\pi_l\|\pi^\theta_l)\le\mathrm{KL}(\xi^\star_l\|\xi^\theta_l)$. With
$\pi^\theta_0=\pi_0$, induction gives $\mathrm{KL}(\pi_l\|\pi^\theta_l)\le\Delta_l$ and $\mathrm{KL}(\xi^\star_l\|\xi^\theta_l)\le\Delta_l$,
and by Pinsker's inequality $\|\pi_l-\pi^\theta_l\|_{\mathrm{TV}}$ and $\|\xi^\star_l-\xi^\theta_l\|_{\mathrm{TV}}$ are at most
$(\Delta_l/2)^{1/2}$.

\emph{Variance.} The lineages are independent and identically distributed whatever the kernels, so the variance formula
of Section~\ref{sec:analysis} and the argument of Theorem~\ref{thm:mse} give
$\Var(\widehat\pi_L(\varphi))\le(\sum_l\tilde s_l/\sqrt{N_l})^2$ as soon as $V_l\le\tilde s_l^{\,2}$. For $l\ge1$,
$V_l\le\|\varphi\|_{\Lip}^2\E_{\xi^\theta_l}\|X-Y\|_2^2$, and since $0\le\|x-y\|_2^2\le4d$ on $\mathsf{X}\times \mathsf{X}$,
\[
  \E_{\xi^\theta_l}\|X-Y\|_2^2\le\E_{\xi^\star_l}\|X-Y\|_2^2+4d\|\xi^\star_l-\xi^\theta_l\|_{\mathrm{TV}}
  \le\kappa W_2^2(\pi_{l-1},\pi_l)+\omega_l+2d(2\Delta_l)^{1/2}.
\]

\emph{Bias.} $\E\widehat\pi_L(\varphi)=\pi^\theta_L(\varphi)$, and the oscillation of $\varphi$ on $\mathsf{X}$ is at most
$2d^{1/2}\|\varphi\|_{\Lip}$, so $|\pi^\theta_L(\varphi)-\pi_L(\varphi)|\le2d^{1/2}\|\varphi\|_{\Lip}(\Delta_L/2)^{1/2}
=\|\varphi\|_{\Lip}(2d\Delta_L)^{1/2}$. The mean square error is the squared bias plus the variance.
\end{proof}
\begin{proof}[Proof of Corollary~\ref{cor:fittedmh}]
We condition on the fitted parameters, which are independent of the particles by \ref{ass:A4}, and write $\widetilde\pi^\theta_l=\pi^\theta_{l-1}K_{\theta_l}$
for the law of the transported particles, so that $\pi^\theta_l=\widetilde\pi^\theta_lM_l^{n_l}$.

\emph{Bias.} Since $\pi_l=\pi_{l-1}K^\star_l$,
\[
  \|\widetilde\pi^\theta_l-\pi_l\|_{\mathrm{TV}}\le\|\pi^\theta_{l-1}K_{\theta_l}-\pi_{l-1}K_{\theta_l}\|_{\mathrm{TV}}
  +\|\pi_{l-1}K_{\theta_l}-\pi_{l-1}K^\star_l\|_{\mathrm{TV}}\le T_{l-1}+(\mathcal E_l/2)^{1/2},
\]
because a Markov kernel does not increase the total variation distance, and the second term is at most
$\|\pi_{l-1}\otimes K_{\theta_l}-\xi^\star_l\|_{\mathrm{TV}}\le(\mathcal E_l/2)^{1/2}$ by Pinsker's inequality. Since $M_l$ leaves $\pi_l$
invariant and has Dobrushin coefficient $\tau_l$, $\|\mu M_l-\pi_l\|_{\mathrm{TV}}\le\tau_l\|\mu-\pi_l\|_{\mathrm{TV}}$ for every
$\mu$, which gives the recursion for $T_l$. Iterating it with $\tau_l^{n_l}\le\bar\tau$ and $\mathcal E_l\le\bar{\mathcal E}$ gives
$T_L\le(\bar{\mathcal E}/2)^{1/2}\sum_{k\ge1}\bar\tau^k$. As in the proof of Corollary~\ref{cor:fitted},
$|\E\widehat\pi_L(\varphi)-\pi_L(\varphi)|=|\pi^\theta_L(\varphi)-\pi_L(\varphi)|\le2d^{1/2}\|\varphi\|_{\Lip}T_L$.

\emph{Variance.} The lineages remain independent and identically distributed, so it suffices to bound $V_l$. With
$\widetilde X_l$ the transported particle and $X_l$ the particle after the Metropolis--Hastings steps,
$\|X_l-X_{l-1}\|_2^2\le2\|\widetilde X_l-X_{l-1}\|_2^2+2\|X_l-\widetilde X_l\|_2^2$. The pair $(X_{l-1},\widetilde X_l)$ has law
$\pi^\theta_{l-1}\otimes K_{\theta_l}$, whose total variation distance to $\xi^\star_l$ is at most
$T_{l-1}+\|\pi_{l-1}\otimes K_{\theta_l}-\xi^\star_l\|_{\mathrm{TV}}\le T_{l-1}+(\mathcal E_l/2)^{1/2}$, and $\|x-y\|_2^2\le4d$ on
$\mathsf{X}\times \mathsf{X}$, so
\[
  \E\|\widetilde X_l-X_{l-1}\|_2^2\le\kappa W_2^2(\pi_{l-1},\pi_l)+\omega_l+4d\big(T_{l-1}+(\mathcal E_l/2)^{1/2}\big).
\]
Hence $V_l\le\|\varphi\|_{\Lip}^2\E\|X_l-X_{l-1}\|_2^2\le\tilde s_l^{\,2}$.
\end{proof}

\section{Implementation Details}\label{app:impl}

\subsection{LightSB fits}\label{app:fits}
At each transition $l-1\to l$ the potential is
$v_\theta(y)=\sum_{j=1}^4\vartheta_j\,\phi_d(y;m_j,\Sigma_j)$ with diagonal $\Sigma_j$, fitted by minimizing the
self-normalized objective $\mathcal L^{N_{l-1}}(\theta)$ of Section~\ref{sec:approach} on the level-$(l-1)$ cloud with Adam,
minibatches of $\min(N_{l-1},1024)$, gradient-norm clipping at 10 and a cosine learning-rate schedule from
$10^{-2}\varepsilon_l/0.05$. The cloud is whitened by its own mean and covariance before fitting, and the fit is initialized at
the identity coupling: $\Sigma_j=\varepsilon_lS_lI_d$ and $m_j=(1-S_l)a_j$ for anchors $a_j$ drawn from the weighted cloud, so
that initially $Y\mid x,j\sim\Na(S_lx+(1-S_l)a_j,(1-S_l^2)I_d)$. Most of the fitting error
sits at the first few transitions, so the number of Adam iterations is set per transition: $(300,300,300,300,100,\dots)$
for the first problem and $(1000,1000,1000,300,100,\dots)$ for the second and third.

\subsection{PDE solves per method}\label{app:solves}
At transition $l-1\to l$ the methods solve the level-$l$ PDE at the following numbers of points (with $n_l$ the number
of MCMC sweeps):
\begin{center}\footnotesize
\begin{tabular}{lll}
\toprule
method & level-$l$ solves & purpose\\
\midrule
MLSMC-pCN (both variants) & $N_{l-1}+n_lN_l$ & weights $\gamma_l/\gamma_{l-1}$; mutation (none at $l=L$)\\
LightSB & $N_{l-1}+N_l$ & training weights $\gamma_l/\gamma_{l-1}$; transported particles\\
TransportMH & $N_{l-1}+N_l+n_lN_l$ & as LightSB; MH sweeps (also at $l=L$)\\
\bottomrule
\end{tabular}
\end{center}
The transported particles are new points, so $\gamma_l$ must be evaluated there, for the MH ratio and for the
training weights of the next transition. LightSB and TransportMH also spend time on the stochastic gradient fits, which involve no PDE
solve. In test case 1 the element averages of $k$ and the load integrals are exact; in test case 2 $k$ is evaluated at
triangle centroids and the systems are solved by sparse LU (dense Cholesky for $l\le1$); in test case 3 the spatial
derivatives are applied through FFTs with the modified wavenumbers $\sin(qh_l)/h_l$ and $\{2\sin(qh_l/2)/h_l\}^2$ for a wavenumber $q$, and a
level-4 solve ($128^2$ grid points, 128 steps) takes about 0.4\,s.

\subsection{Sample sizes and reference values}\label{app:settings}
\paragraph{Test case 1.} $\alpha=1$, $\beta=2$, $\zeta=1$, horizons $L=1,\dots,9$, $N_0=10^5$ at $L=9$:
\[
  L=9:\ N=(100000,\,35356,\,12500,\,4420,\,1563,\,553,\,196,\,100,\,100,\,100).
\]
Level-0 samples are i.i.d.\ from $\pi_0$ by rejection from the prior. The reference value of $\pi(\varphi)$ is
$50.33335\pm0.00010$, from self-normalized importance sampling with $2\times10^8$ samples on the continuum model, using a
Student-$t$/prior defensive mixture proposal with bounded weights.

\paragraph{Test case 2.} Within a run $N_l\propto h_l^{3.25}$ ($\beta=4$ from the $h_l^2$ convergence of the sensor values,
$\zeta=2.5$ from the solve time), floored at 100, and $N_0$ grows by a factor 4 per added level:
\begin{align*}
  L=5&:\ N=(20000,\,2103,\,221,\,100,\,100,\,100),\\
  L=6&:\ N=(80000,\,8409,\,884,\,100,\,100,\,100,\,100).
\end{align*}
Prior rejection is not feasible at level 0, so the level-0 samples come from an adaptive tempered SMC sampler on
$\pi_0$ (likelihood tempering with ESS 50\%, five pCN-MH sweeps per stage). There is no continuum reference; the target is
$\pi_5(\varphi)$, with reference value $0.99780\pm0.00112$ from 40 independent runs of method (ii) at ten times the
$L=5$ allocation. The level-6 correction is $(1\pm1)\times10^{-5}$, far below this standard error, so the $L=6$ runs
are scored against the same value.

\paragraph{Test case 3.} $N_l\propto h_l^{3.5}$ ($\beta=4$ from the $h_l^2$ convergence of the sensor values, $\zeta=3$
from the solve cost), floored at 100, with $N_0$ growing by a factor 4 per added level:
\[
  L=4:\ N=(20000,\,1768,\,156,\,100,\,100),
\]
and $(5000,442,100,100)$, $(1250,110,100)$, $(312,100)$ for $L=3,2,1$. The level-0 samples come from the tempered SMC
sampler of test case 2. The target is $\pi_3(\varphi)$, with reference value $0.7981\pm0.0029$ from 20 independent runs
of method (ii) with $N=(20000,1768,400,400)$, independent of the runs being scored. The level-4 correction is at most
$6\times10^{-4}$ in absolute value for every method, below this standard error, so the $L=4$ runs are scored against the
same value. The reference standard error is 0.36\% of $\bar\pi$, which is not negligible next to the smallest errors
in Table~\ref{tab:results}: it contributes about a sixth of the mean squared error of TransportMH at $L=4$, and it is included in the
standard errors of the relative bias.

\subsection{Breakdown of the computing time}\label{app:cost}
Table~\ref{tab:cost} splits the cost at the finest horizon. Three points explain the runtime differences in
Figure~\ref{fig:err-time}.
\begin{table}[!htbp]
\centering\footnotesize
\setlength{\tabcolsep}{3pt}
\begin{tabular}{l rrrr rrrr rrrr}
\toprule
& \multicolumn{4}{c}{1D elliptic, $L=9$} & \multicolumn{4}{c}{Darcy flow, $L=6$} & \multicolumn{4}{c}{Navier--Stokes, $L=4$}\\
\cmidrule(lr){2-5}\cmidrule(lr){6-9}\cmidrule(lr){10-13}
method & solves & work & PDE & fit & solves & work & PDE & fit & solves & work & PDE & fit\\
 & ($10^5$) & ($10^7$) & (s) & (s) & ($10^5$) & ($10^7$) & (s) & (s) & ($10^4$) & ($10^8$) & (s) & (s)\\
\midrule
MLSMC-pCN    & 4.2  & 1.2  & 2.9\,(1.2)  & --  & 2.0  & 5.4  & 35.9\,(10.6)  & --  & 4.8 & 5.2  & 49.1\,(10.5)  & --\\
+ ESS sweeps & 13.7 & 2.8  & 9.4\,(7.6)  & --  & 4.25 & 8.3  & 46.9\,(21.5)  & --  & 9.6 & 7.1  & 81.0\,(41.7)  & --\\
LightSB      & 3.1  & 0.84 & 2.2         & 2.8 & 1.8  & 4.2  & 29.0          & 3.0 & 4.4 & 4.4  & 42.8          & 2.9\\
TransportMH  & 14.3 & 3.1  & 9.7\,(8.0)  & 2.7 & 4.35 & 14.1 & 114.0\,(88.7) & 3.0 & 9.8 & 15.9 & 142.7\,(103.3) & 3.0\\
\bottomrule
\end{tabular}
\caption{Cost per run at the finest horizon (means over 30 runs): number of PDE solves; element work
$\sum_l(\text{level-}l\text{ solves})\times(\text{elements on level }l)$ (in test case 3, grid points times time steps);
wall time of the PDE solves, with in parentheses the part inside MCMC (for TransportMH this includes the solves at the
transported particles; LightSB has no MCMC); and wall time of the LightSB fits. The total wall times are those of Table~\ref{tab:results};
the remainder is resampling, the estimator and bookkeeping.}\label{tab:cost}
\end{table}

\begin{enumerate}[label=(\arabic*)]
\item \emph{The 30 sweeps at the first transition dominate the solve count} of both ESS-adaptive methods: they are run
  on $N_1=35\,356$ particles in test case 1 and $8\,409$ in test case 2, i.e., $1.06\times10^6$ of the $1.37\times10^6$
  solves of MLSMC-pCN with ESS sweeps in test case 1 and $2.5\times10^5$ of $4.25\times10^5$ in test case 2. Level-1 solves
  are cheap, but there are many: they take 7.2\,s in test case 1 and 12\,s in test case 2. Plain MLSMC-pCN does not pay
  this, which is why it is the fastest method at equal $L$; it pays in accuracy at level 2, whose correction standard
  deviation is $0.0061$ against $0.0011$ with the sweeps (test case 1).
\item \emph{TransportMH and MLSMC-pCN with ESS sweeps make almost the same number of solves} ($1.43$ against
  $1.37\times10^6$; $4.35$ against $4.25\times10^5$), but the extra solves of TransportMH, $N_l$ transported particles plus
  $3N_l$ MH solves at every level including $L$, sit at the finest and most expensive levels. In test case 1 every solve
  is cheap (a tridiagonal system), the element work differs by only 13\%, and most of the time difference is the
  LightSB fitting: 2.7\,s per run, a fifth of the total; it hardly depends on $N_l$, because the minibatches are capped at
  1024, so it weighs most when the solves are cheap. In test case 2 the fits cost 3.0\,s (2.5\%), but a level-6 solve takes about 0.16\,s and TransportMH makes
  400 of them (100 transported particles and 300 MH solves) where MLSMC, which does not mutate at the final level, makes
  none beyond the 100 for its weights. The transported particles and MH sweeps at levels 5 and 6 take 74\,s of the
  119\,s of TransportMH at $L=6$; the rest, mainly the $\gamma_l/\gamma_{l-1}$ solves (25\,s), the level-1 sweeps (12\,s) and
  the fits (3\,s), matches MLSMC-pCN with ESS sweeps. Test case 3 is similar: the two methods make $9.8$ and
  $9.6\times10^4$ solves, but TransportMH makes 500 level-4 solves per run (100 for the weights, 100 at the transported
  particles and 300 MH solves) against 100 for MLSMC, and at $2.1\times10^6$ grid-point steps per level-4 solve these 400
  extra solves are $8.4\times10^8$ of the $8.8\times10^8$ difference in work. The fits cost 3.0\,s, 2\% of the total.
\item \emph{The fine-level MH sweeps of TransportMH move little.} Where $e_l\approx1$, $\rho_l=S_l$ and
  $1-\rho_l^2=\varepsilon_lS_l\approx\varepsilon_l\le2\times10^{-4}$ from level 5 on, and the acceptance rate is 94--99\% (in test
  case 3, $\varepsilon_4=7.8\times10^{-4}$ and the acceptance rate at level 4 is about 96\%). Each sweep
  then adds about as much noise to $X_l^i-X_{l-1}^i$ as the transport itself, so three sweeps multiply the variance of
  the paired differences $\varphi(X_l^i)-\varphi(X_{l-1}^i)$ by about four: $2.3\times10^{-6}$ before and
  $9.4\times10^{-6}$ after MH at level 6 of test case 2, $2.5\times10^{-6}$ and $9.7\times10^{-6}$ at level 9 of test case
  1. At the coarse transitions the sweeps are what make TransportMH accurate (without them, LightSB has a relative bias of
  $-2\times10^{-3}$ and $-3\times10^{-2}$), but at fine levels they cost three expensive solves per particle and inflate
  the variance. Skipping them where $e_l\ge e_{\mathrm{high}}$ would, from the measured solve times, bring the $L=6$ run of
  test case 2 from 119\,s to roughly 60--65\,s; we have not tested what it does to the bias.
\end{enumerate}

\end{document}